\documentclass[reqno]{amsart}
\usepackage{amsmath,amssymb,amsthm,amscd,amsfonts,mathrsfs,verbatim}
\usepackage[all]{xy}
\usepackage{color}
\usepackage{epsfig}

\newcommand{\N}{{\mathbb N}}                   
\newcommand{\Z}{{\mathbb Z}}                   
\newcommand{\R}{{\mathbf R}}                   
\newcommand{\Q}{{\mathbb Q}}                   
\newcommand{\C}{{\mathbb C}}                   
\renewcommand{\H}{{\mathbf H}}                   
\newcommand{\Mod}{{\mathcal M}}               
\renewcommand{\O}{{\mathcal O}}
\newcommand{\CP}[1]{\mathbb{C}P^{#1}}      
\newcommand{\E}{{\mathcal E}}

\newcommand{\blue}[1]{#1}

\DeclareMathOperator{\Ker}{Ker}
\DeclareMathOperator{\Coker}{Coker}
\DeclareMathOperator{\im}{Im}
\DeclareMathOperator{\Gr}{Gr}
\DeclareMathOperator{\Dol}{Dol}
\DeclareMathOperator{\Betti}{B}
\DeclareMathOperator{\dR}{dR}
\DeclareMathOperator{\GM}{GM}
\DeclareMathOperator{\reg}{reg}
\DeclareMathOperator{\degen}{deg}

\DeclareMathOperator{\Res}{Res}
\DeclareMathOperator{\ad}{ad}

\newtheorem{lem}{Lemma}
\newtheorem{theorem}{Theorem}
\newtheorem{prop}{Proposition}
\newtheorem{remark}{Remark}
\newtheorem{condition}{Condition}

\title[$P = W$ for Painlev\'e spaces]{Perversity equals weight for Painlev\'e spaces}

\author{Szil\'ard Szab\'o}
\address{Budapest University of Technology and Economics, 1111. Budapest,
Egry J\'ozsef utca 1. H \'ep\"ulet, Hungary, and 
R\'enyi Institute of Mathematics, 1053. Budapest, Re\'altanoda
utca 13-15. Hungary}
\email{szabosz@math.bme.hu, szabo.szilard@renyi.mta.hu}

\begin{document}

\begin{abstract}
We provide further evidence to the $P=W$ conjecture of de Cataldo, Hausel and Migliorini, by checking it in the Painlev\'e cases. 
Namely, we compare the perverse Leray filtration induced by the Hitchin map on the cohomology spaces of the Dolbeault moduli space 
and the weight filtration on the cohomology spaces of the irregular character variety corresponding to each of the Painlev\'e $I-VI$ systems. 
We find that the two filtrations agree. 
Along the way, we prove the Geometric $P=W$ conjecture of Katzarkov, Noll, Pandit and Simpson in the Painlev\'e cases, and show that 
in these cases the Geometric $P=W$ conjecture implies the $P=W$ conjecture. 
\end{abstract}

\maketitle

\section*{Introduction and statement of main result}
Throughout the paper we let $X$ denote one of the symbols 
$$
  I, II, III(D6), III(D7), III(D8), IV, V_{\degen}, V, VI
$$
so that the index $PX$ refers to Painlev\'e $X$. 
In \cite{PS}, irregular Betti moduli spaces (also called wild character varieties following \cite{Boalch2}) $\Mod_{\Betti}^{PX}$ are defined and shown to be $\C$-analytically isomorphic 
under the Riemann--Hilbert correspondence to irregular de Rham spaces $\Mod_{\dR}^{PX}$. 
(\blue{At a higher level of generality, moduli spaces of untwisted irregular connections of arbitrary rank on a compact Riemann surface of arbitrary genus were constructed in \cite{Boalch1} 
as algebraic symplectic manifolds, and the irregular Riemann--Hilbert correspondence for the moduli spaces was proven in \cite{Boalch3}, building on the 
categorical correspondence of Malgrange \cite[Chapitre~4]{Mal}.}) 
It follows from \cite[Theorem~1]{BB} that for every $X$ in the so-called untwisted cases $II, III(D6), IV, V, VI$ smooth complex analytic moduli spaces 
$\Mod_{\Dol}^{PX}$ exist and are diffeomorphic under non-abelian Hodge theory to the corresponding $\Mod_{\dR}^{PX}$. 
A combination of these results implies that in the untwisted cases $\Mod_{\Betti}^{PX}$ and $\Mod_{\Dol}^{PX}$ are diffeomorphic; 
such a diffeomorphism is expected to exist in the remaining (twisted) cases too. 
In \cite{ISS1}, \cite{ISS2} we gave an explicit description of the spaces $\Mod_{\Dol}^{PX}$ and of their (irregular) Hitchin map in terms of elliptic pencils. 
In \cite{PS}, an explicit description of the spaces $\Mod_{\Betti}^{PX}$ is provided as affine cubic surfaces. 

Deligne \cite{DelHodge2} constructs a weight filtration $W$ on the complex cohomology spaces of an affine algebraic variety. 
In particular, the cohomology spaces of $\Mod_{\Betti}^{PX}$ carry a mixed Hodge structure. 
On the other hand, the Hitchin map endows the complex cohomology spaces of $\Mod_{\Dol}^{PX}$ with a perverse Leray filtration $P$ \cite{BBD}. 
Following \cite[page 2]{HMW}, we set 
\begin{align*}
   PH^{PX}(q,t) & = \sum_{i,k} \dim_{\Q} \Gr^P_i H^k (\Mod_{\Dol}^{PX}, \Q ) q^i t^k, \\
   WH^{PX}(q,t) & = \sum_{i,k} \dim_{\C} \Gr^W_{2i} H^k (\Mod_{\Betti}^{PX}, \C ) q^i t^k.
\end{align*}
Remarkably, in the rank $2$ case without (regular or irregular) singularities equality between these two polynomials for Dolbeault and Betti spaces 
corresponding to each other under non-abelian Hodge theory and Riemann--Hilbert correspondence was 
proven in  \cite[Theorem~1.0.1]{HdCM}, and conjectured to be the case in general (the $P=W$ conjecture). 

The perverse filtration for some logarithmic Hitchin systems was studied by Z.\;Zhang \cite{Zhang}, where he showed multiplicativity of the filtration with respect to wedge product on 
Hilbert schemes of smooth projective surfaces fibered over a curve, and thereby computed their perverse polynomials. 
More generally, W.\;Chuang, D.\;Diaconescu, R.\;Donagi and T.\;Pantev conjectured a formula for the perverse Hodge polynomial of moduli spaces of meromorphic Higgs bundles with one irregular singularity \cite{CDDP}.
On the Betti side, T.\;Hausel, M.\;Mereb and M.\;Wong investigated the weight filtration on the cohomology of character varieties of punctured curves with one irregular singularity, 
and extended the $P=W$ conjecture to this case \cite[Problem~0.1.4]{HMW}.
The purpose of this paper is to give an affirmative answer to this conjecture in the Painlev\'e cases. 
Notice that not all the cases we study fall into the class studied in \cite{HMW}, because some of them admit two irregular singularities, some of which with twisted formal type. 
\begin{theorem}\label{thm:PW}
 For every $X$ the perverse Leray and weight 
 polynomials on the cohomology of the Dolbeault and Betti spaces mapped to each other by non-abelian Hodge theory and Riemann--Hilbert correspondence 
 agree: 
 \begin{equation}\label{eq:PW}
  PH^{PX}(q,t) = WH^{PX}(q,t). 
\end{equation}
\end{theorem}

Moreover, 
we will prove that the classes generating the exotic pieces of the $P$ and $W$
filtrations match up under non-abelian Hodge theory. 
The idea of the proof is to pass through establishing a conjecture of C.\;Simpson \cite[Conjecture~11.1]{Sim} in this special case (c.f. Theorem~\ref{thm:Simpson}). 
Matching of the $P$ and $W$ filtrations on $H^2$ amounts roughly speaking to the Poincar\'e dual of Theorem~\ref{thm:Simpson} in the boundary $3$-manifold of a neighbourhood of infinity in the moduli space. 

\begin{theorem}\label{thm:Simpson}
 There exists a smooth compactification $\widetilde{\Mod}_{\Betti}^{PX}$ of $\Mod_{\Betti}^{PX}$ by a simple normal crossing divisor $D$ such that the body 
 $| \mathcal{N}^{PX} |$ of the nerve complex $\mathcal{N}^{PX}$ of $D$ is homotopy equivalent to $S^1$. 
 Moreover, for some sufficiently large compact set $K \subset {\Mod}_{\Betti}^{PX}$, there exists a homotopy commutative square  
 $$
  \xymatrix{\Mod_{\Dol}^{PX} \setminus K  \ar[d]_h \ar[r]^{\psi} & {\Mod}_{\Betti}^{PX} \setminus K \ar[d]^{\phi} \\
  D^{\times} \ar[r] & | \mathcal{N}^{PX} |. }
 $$
 Here, $h$ denotes the Hitchin map, $D^{\times} \subset Y$ is a neighbourhood of $\infty$ in the Hitchin base, and the top row is 
 the diffeomorphism coming from non-abelian Hodge theory. 
\end{theorem}
For details, see Section~\ref{sec:proof}.
This statement in higher generality was conjectured by L.\;Katzarkov, A.\;Noll, P.\;Pandit and C.\;Simpson~\cite[Conjecture 1.1]{KNPS}; in their generalization the homotopy type 
of the body of the nerve complex is conjectured to be that of a sphere and a similar homotopy commutativity relation is conjectured to hold. 
An analogous statement to the homotopy sphere assertion has been proven by A.\;Komyo \cite{Komyo} for two $2$-dimensional and a $4$-dimensional logarithmic Dolbeault moduli spaces. 
In~\cite{Sim}, C.\;Simpson proved the homotopy sphere assertion in the more general setup of character varieties with an arbitrary number of punctures on the projective line, 
and named the homotopy commutativity assertion the Geometric $P = W$ conjecture~\cite[Conjecture 11.1]{Sim}.

In Section~\ref{sec:prep} we give some background material necessary to understand our constructions. 
In Section~\ref{sec:perverse} we describe the perverse filtration on $\Mod_{\Dol}^{PX}$. 
In Section~\ref{sec:weight} we determine the weight filtration on ${\Mod}_{\Betti}^{PX}$ and prove Theorem~\ref{thm:PW}. 
In Section~\ref{sec:proof} we prove Theorem~\ref{thm:Simpson}, 
and in Section~\ref{sec:filtrations} we prove equality of the $P$ and $W$ filtrations in the Painlev\'e VI case.

\bigskip

\noindent {\bf Acknowledgements:} During the preparation of this document, the author benefited of discussions with 
  P.\;Boalch, T.\;Hausel, R.\;Mazzeo, T.\;Mochizuki,  M-H.\;Saito, C.\;Simpson and A.\;Stipsicz, and was supported by the \emph{Lend\"ulet} Low Dimensional Topology 
  grant of the Hungarian Academy of Sciences and by the grants K120697 and KKP126683 of NKFIH. 
  I would also like to thank the anonymous referee for useful suggestions.
  Finally, I would like to thank the hospitality of Kobe University where 
  the first draft of this article was written.
  
\section{Preparatory material}\label{sec:prep}

\subsection{Meromorphic Higgs fields}\label{ssec:Higgs}

Let $D = \sum m_i p_i$ be an effective divisor on $\CP1$, where $p_i \in \CP1$ and $m_i \in \Z_{+}$ satisfy $\sum m_i = 4$. 
The moduli spaces $\Mod_{\Dol}^{PX}$ parameterize certain rank $2$ meromorphic parabolic Higgs bundles over $\CP1$ with poles at such a divisor $D$. 
A \emph{rank $2$ meromorphic Higgs bundle over $\CP1$ with poles at $D$} is a couple $(\mathcal{E}, \theta)$ where 
\begin{itemize}
 \item $\mathcal{E}$ is a rank $2$ holomorphic vector bundle over $\CP1$,
 \item $\theta$ is a meromorphic $\mathcal{O}$-linear morphism 
$$
  \theta \colon \mathcal{E} \to \mathcal{E} \otimes K_{\CP1} (D)
$$
\end{itemize}
subject to certain conditions that we will spell out below. We will assume that the degree of $\mathcal{E}$ is odd. 
With respect to some holomorphic trivialization of $\mathcal{E}$ near $p_i$, we may expand $\theta$ as a convergent Laurent series
$$
  \sum_{m=-m_i}^{\infty} A_m (z-p_i)^{m}
$$
for some $A_m\in \mathfrak{gl}_2(\C)$. 
We define the \emph{residue} of $\theta$ at $p_i$ as $\Res_{p_i} \theta = A_{-1}$; the residue is well-defined up to adjoint action of $\mbox{Gl}_2(\C)$. 
In case $m_i = 1$ the point $p_i$ is said to be a \emph{logarithmic} singular point of $\theta$; if on the other hand $m_i > 1$ we say that the singularity 
of the Higgs field is irregular, and we define its \emph{irregular part} as 
$$
  \sum_{m=-m_i}^{-2} A_m (z-p_i)^{m},
$$
We will assume the following genericity condition for the irregular parts: 
\begin{condition}\label{cond:regular}
For all irregular singular points, the leading-order term $A_{-m_i}$ of the expansion of $\theta$ at $p_i$ is a regular endomorphism. 
\end{condition}
This condition is manifestly independent of the chosen trivialization. 
One further needs to distinguish between Higgs fields with twisted and untwisted irregular part. 
We say that $\theta$ has \emph{untwisted} irregular singularity at $p_i$ if the leading-order term $A_{-m_i}$ of the expansion of $\theta$ is regular semi-simple, and \emph{twisted} irregular singularity at $p_i$ 
if its leading-order term is (up to adding a multiple of the identity) a regular nilpotent. 
If $\theta$ has an untwisted irregular singular point at $p_i$, then we will continue to write $n_i = m_i$ for the coefficient of $p_i$ in the divisor $D$. 
If, on the other hand, $\theta$ has a twisted irregular singular point at $p_i$, then we will denote its coefficient in $D$ as $n_i = m_i - \frac 12$. 
To sum up, the notation
$$
  \sum n_i p_i
$$
encodes the order and twistedness of the poles $p_i$ of $\theta$, and it varies in function of the Painlev\'e type according to Table~\ref{table:singular}. 
\begin{table}
\begin{center}
\begin{tabular}{|l|l|r|}
 \hline 
 $X$ &  $D= \sum n_i p_i$ \\
 \hline 
 \hline 
 $VI$ & $p_1 + p_2 + p_3 + p_4$ \\
 \hline 
 $V$ & $2p_1 + p_2 + p_3$ \\ 
 \hline 
 $III(D6) = V_{\degen}$ & $2 p_1 + 2 p_2 ; \frac 32 p_1 + p_2 + p_3$ \\ 
 \hline 
 $III(D7)$ & $\frac 32 p_1 + 2 p_2$ \\
 \hline 
 $III(D8)$ & $\frac 32 p_1 + \frac 32 p_2$ \\ 
 \hline 
 $IV$ & $3 p_1 + p_2$ \\
 \hline 
 $II$ & $4 p_1; \frac 52 p_1 + p_2$ \\ 
 \hline 
 $I$ & $\frac 72 p_1$ \\
 \hline 
\end{tabular}
\end{center}
 \caption{Types of singularities of the Higgs field. In case several possibilities give rise to the same Painlev\'e equation, they are separated by a semi-colon.} 
 \label{table:singular}
\end{table}
Let us emphasize again that in case the order of the pole of $\theta$ at $p_i$ is equal to $1$, then in the form of $D$ we do not distinguish between regular semi-simple residue orbits and ones with two equal eigenvalues. 
Under Condition~\ref{cond:regular} a \emph{quasi-parabolic structure} is the datum of a $1$-dimensional subspace $\ell_i \subset \mathcal{E}|_{p_i}$ in the $2$-dimensional fiber of $\mathcal{E}$ over $p_i$ for each $i$ such that $m_i =1$.
We impose the following compatibility condition on the above data:
\begin{condition}\label{cond:compatible}
 The line $\ell_i$ is an eigenspace of $\Res_{p_i} \theta$.
\end{condition}
Under the above conditions, in order to define the moduli spaces $\Mod_{\Dol}^{PX}$ we need to fix 
\begin{itemize}
 \item the divisor $D = \sum n_i p_i$ where $2n_i \in \Z$,
 \item for each $i$ the irregular part of $\theta$ near $p_i$ and the adjoint orbit of $\Res_{p_i} \theta$, 
 \item for all $i$ such that $m_i = 1$ the orbit of the reduction of $\Res_{p_i} \theta$ to the Levi quotient of the parabolic subalgebra $\mathfrak{p}_i \subset \mathfrak{gl}_2(\C)$ defined by the line $\ell_i$. 
\end{itemize}
In concrete terms, the requirement on Levi quotient means that the eigenvalues of $\Res_{p_i} \theta$ are fixed, but their adjoint orbit in $\mathfrak{gl}_2(\C)$ is not. 
The moduli spaces $\Mod_{\Dol}^{PX}$ parameterize triples $(\mathcal{E}, \theta, \{ \ell_i \})$ where $(\mathcal{E}, \theta)$ is a meromorphic Higgs bundle over $\CP1$ with polar divisor bounded from above by $D$, 
such that 
\begin{itemize}
 \item for all $i$ satisfying $m_i > 1$ the irregular part of $\theta$  with respect to some holomorphic trivialization of $\mathcal{E}$ near $p_i$ is fixed,
 \item for all $i$ satisfying $m_i > 1$ the adjoint orbit of $\Res_{p_i} \theta$  is fixed, 
 \item and for all $i$ satisfying $m_i = 1$ the line $\ell_i$ defines a compatible quasi-parabolic structure such that the adjoint orbit of the reduction of $\Res_{p_i} \theta$ to the Levi quotient of $\mathfrak{p}_i$ is fixed.  
\end{itemize}
In order to get coarse moduli spaces corresponding to Artin stacks, one needs to restrict to objects satisfying a parabolic semi-stability condition. 
We assume that the parabolic weights are general, so that stability is equivalent to semi-stability. 
For further details, we refer to~\cite{ISS1,ISS2,ISS3}. Let us call attention to the following difference between our assumptions in the present paper and those of~\cite{ISS1,ISS2,ISS3}, 
that will play a fundamental role in the analysis. 
Namely, in~\cite{ISS1,ISS2,ISS3}, we only considered singular Higgs fileds whose residues at logarithmic points were regular, i.e. had non-trivial nilpotent part in case of equal eigenvalues. 
As opposed to this, in case the residue $\Res_p (\theta )$ of the Higgs field at some logarithmic point $p$ is assumed to have two equal eigenvalues, then in the present paper we consider the moduli space 
$\Mod_{\Dol}^{PX}$ of corresponding Higgs bundles completed with all Higgs bundles having the same eigenvalues of its residue but with trivial nilpotent part, equipped with a quasi-parabolic structure of 
full flag type at these points (the compatibility condition being vacuous). The reason we consider this completion of our previously studied spaces is that the Hitchin fibers of the non-completed moduli spaces may be non-compact, as endomorphisms 
with non-trivial nilpotent part may converge to ones with trivial nilpotent part. 

\blue{
Importantly for our purposes, we have: 
\begin{lem}\label{lem:smoothness}
 The completed moduli space $\Mod_{\Dol}^{PX}$ is a smooth complex manifold, and the irregular Hitchin map 
\begin{equation}\label{eq:Hitchin_map}
   h: \Mod_{\Dol}^{PX} \to Y=\C 
\end{equation}
is proper. 
\end{lem}
}
\begin{proof}
\blue{
 The proof of the first statement follows from \cite[Theorem~5.4]{BB}. 
 Indeed, let us consider an endomorphism 
 $$
  A  \in \mathfrak{gl} ( E|_p ) 
 $$
 of the fiber of a given rank $2$ smooth vector bundle $E$ at $p$. Let the decomposition of $A$ into semi-simple and nilpotent part be 
 $$
  A = A^{s} + A^{nil}
 $$
 and assume that $A^{nil} \neq 0$ (so necessarily $A^{s}$ is a multiple of identity). Finally, let 
 $$
   \mathfrak{p}\subset \mathfrak{gl} ( E|_p )
 $$ 
 stand for the parabolic subalgebra containing $A^{nil}$ and 
 $$
  \pi: \mathfrak{p} \to \mathfrak{l}
 $$ 
 its Levi quotient.
 It follows from \cite[Theorem~5.4]{BB} that the moduli space parameterizing irregular Higgs bundles $(\mathcal{E} , \theta)$ 
 endowed with a compatible parabolic structure, with fixed underlying smooth vector bundle $E$ and such that $\pi ( \Res_p (\theta ) ) = \pi (A)$ 
 is a smooth complex manifold. 
 Now, given that $\pi(A^s ) = \pi (A)$ we get that the completed Dolbeault moduli space is smooth.
 }

 \blue{
 Properness follows from \cite{Bottacin, Markman} for the moduli space of Higgs bundles 
 with some poles of total order $n$ over any compact Riemann surface $C$, without any condition on the polar parts and residues at these points. 
 In the case $C = \CP1$ and a divisor of total multiplicity $4$, the base  $\C^8$ of the Hitchin map for this system contains
 the image of those Higgs bundles having prescribed polar parts and residues as an affine open subspace $\mathbb{A}\cong \C$. 
 Namely, $\mathbb{A}$ is specified by the jet of the characteristic coefficients at the punctures of given order (see \cite{ISS1, ISS2, ISS3}, or in greater generality \cite[Theorems~5,~6]{BKV}). 
 The preimage $h^{-1}(a)$ of any $a\in\mathbb{A}$ is the set of all Higgs bundles having characteristic polynomial corresponding to $a$. 
 By \cite[Lemma~10.1]{ISS2}, at any logarithmic singularity $p$ the characteristic polynomial of the residue of the Higgs field is prescribed by $a$, but its adjoint orbit is not. 
 Conversely, if a sequence of Higgs bundles $(\mathcal{E}_n, \theta_n)_{n\geq 1}$ in $h^{-1}(a)$ converge to some Higgs bundle $(\mathcal{E}_0, \theta_0 )$ then the 
 residue of $\theta_0$ at $p$ has the same characteristic polynomial as the residues of $\theta_n$ at $p$. 
 Replacing a finite number of points in the fiber by projective lines (corresponding to choices of a parabolic line $\ell \subset E|_p$) does not modify properness.
 This implies properness for the completed moduli problem.
 }
\end{proof}

\blue{
\begin{remark}
 If $(\mathcal{E}_0 , \theta_0)$ is an irregular parabolic Higgs bundle such that 
 $$
  \Res_p (\theta_0 )  = A  \in \mathfrak{gl} (E|_p )
 $$ 
 with $A^{nil} \neq 0$ then by virtue of Condition~\ref{cond:compatible} the compatible quasi-parabolic line $\ell \subset E|_p$ is uniquely determined by the requirement $A^{nil} \in \mathfrak{p}$. 
 On the other hand, if $(\mathcal{E}_1 , \theta_1)$ is an irregular parabolic Higgs bundle such that 
 $$
  \Res_p (\theta_1 )  = A^s
 $$
 then the compatible quasi-parabolic line $\ell \subset E|_p$ may be chosen arbitrarily. This fact plays a crucial role in the proof of Lemma~\ref{lem:embedding}. 
\end{remark}
}

\subsection{Stokes data}

\blue{Non-abelian Hodge theory \cite{BB} associates a meromorphic connection $(E, \nabla)$ on a holomorphic vector bundle to a meromorphic Higgs bundle, so 
that the irregular part of $\nabla$ is (up to a scalar factor of $2$) the same as the one of the Higgs field. The eigenvalues of the residue and the parabolic weights 
transform according to Simpson's table \cite{Sim_Hodge}, while their nilpotent parts agree. These rules then completely determine the local singularity behaviour of 
$(E, \nabla)$ at the points $p_i$. The irregular Riemann--Hilbert correspondence \cite{Mal} associates the local system of solutions to the meromorphic connection $(E, \nabla)$. 
It can be conveniently encoded in terms of a representation of the fundamental group of the punctured Riemann surface, complemented with Stokes matrices 
satisfying certain conditions.}

\blue{Let us now describe in detail the types of the Stokes local systems associated to the meromorphic connections $(E, \nabla)$ relevant in the Painlev\'e cases following~\cite{PS}. 
To any logarithmic singular point $p_i$ one associates the local holonomy automorphism $T_i\in \mbox{Gl}(2, \C)$ of $(E, \nabla)$ along a loop winding around $p_i$ once in positive direction. 
Assuming the different eigenvalues of $\Res_{p_i} (\nabla )$ do not differ by integers, this holonomy is determined up to conjugacy by the residue of $\nabla$ 
\begin{equation}
   T_i = L_i \exp (2 \pi \sqrt{-1} \Res_{p_i} (\nabla )) L_i^{-1}
\end{equation}
for some $L_i\in \mbox{Gl}(2, \C)$. 
To an untwisted irregular singular point $p_i$ with $m_i > 1$ the meromorphic connection $(E, \nabla)$ determines Stokes matrices 
$$
S_i^1, S_i^2, \ldots, S_i^{2 m_i-3}, S_i^{2 m_i-2}
$$ 
such that $S_i^{2j-1} \in B_i^+$ and $S_i^{2j} \in B_i^-$ where $B_i^{\pm}$ is a pair of opposite Borels. In terms of the local system, the Stokes matrices arise as follows: on overlapping angular sectors 
$U_i^l$ of opening $\frac{2\pi}{2 m_i-2} + \varepsilon$ centered at $p_i$ the connection $\nabla$ admits convergent fundamental systems having growth behaviour predicted by its irregular type; 
the Stokes matrix $S_i^l$ is then the change of trivializations from the fundamental system on $U_i^l$ to the one on $U_i^{l+1}$ (where $l$ is understood modulo $2 m_i-2$). 
We define the formal monodromy $\gamma$ of $\nabla$ as  
$$
  \gamma = \exp( 2 \pi \sqrt{-1} \mbox{diag}\Res_{p_i} (\nabla ) )
$$ 
where $\mbox{diag}\Res_{p_i} (\nabla )$ stands for the diagonal part  of $\Res_{p_i} (\nabla )$ with respect to a basis singled out by the Cartan $B_i^+ \cap B_i^-$.
The holonomy automorphism about the untwisted irregular singular point $p_i$ is then given by 
$$
  T_i = L_i \gamma S_i^{2 m_i-2} S_i^{2 m_i-3} \cdots S_i^2 S_i^1 L_i^{-1}
$$
for some link automorphism $L_i\in \mbox{Gl}(2, \C)$.}

\blue{
For a twisted irregular singular point $p_i$ with $m_i > 1$, the formal monodromy is necessarily equal to 
$$
  \gamma = \begin{pmatrix}
   0 & -1 \\
   1 & 0
  \end{pmatrix},
$$
and we again have Stokes matrices 
$$
  S_i^1, S_i^2, \ldots, S_i^{2 n_i-3}, S_i^{2 n_i-2}
$$
which alternately belong to $B_i^{\pm}$. The holonomy automorphism about the twisted irregular singular point $p_i$ is then equal to 
$$
  T_i = L_i \gamma S_i^{2 n_i-2} S_i^{2 n_i-3} \cdots S_i^2 S_i^1 L_i^{-1}
$$
for some link automorphism $L_i\in \mbox{Gl}(2, \C)$.
}

\blue{The global condition on the above data is that the product of all holonomies 
\begin{equation}\label{eq:holonomy_product}
   \cdots T_2 T_1 = \mbox{I} 
\end{equation}
is the identity. Moreover, there exists a natural $\mathbb{G}_m^k$-action on the above data, where $\mathbb{G}_m = \C^{\times}$ is the multiplicative group of $\C$ and $k\geq 0$ is an integer. 
The wild character variety $\Mod_{\Betti}^{PX}$ is then by definition the GIT-quotient of the $\mathbb{G}_m^k$-action on the set containing all possible Stokes and monodromy data and link automorphisms satisfying~\eqref{eq:holonomy_product}. 
}

\subsection{Singular fibers of elliptic surfaces}\label{ssec:elliptic}

\begin{figure}[hb] 
\begin{center}
\includegraphics[width=10cm]{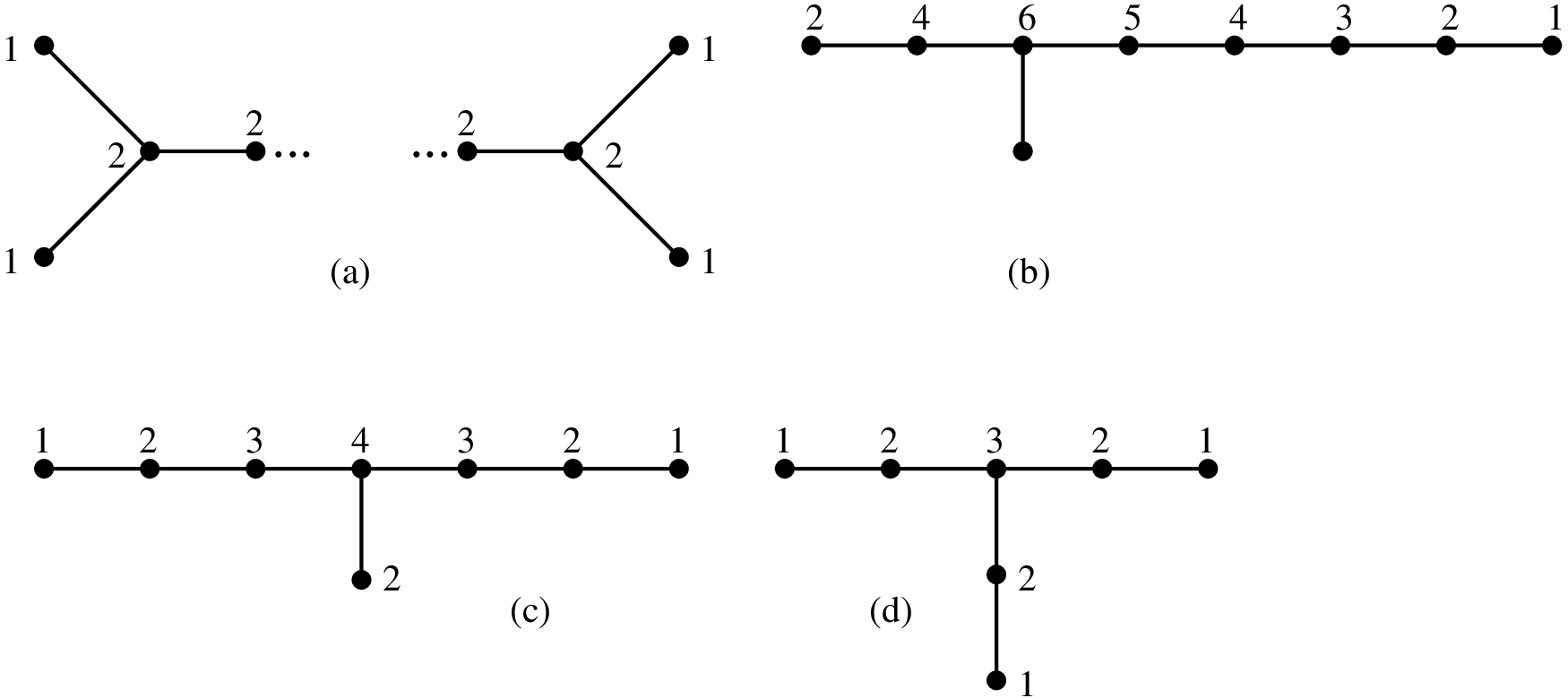} 
\end{center} 
\caption{\quad Dual graphs of singular fibers of types (a) $I_{n-4}^*$, (b) $E_8^{(1)}$, (c) $E_7^{(1)}$, and
(d) $E_6^{(1)}$.  Integers next to vertices indicate the multiplicities of the corresponding components in the fiber. All
vertices of the graphs correspond to rational curves with self-intersection $-2$, and edges correspond to transverse intersection 
of the corresponding components.} 
\label{fig:sing}
\end{figure}

\blue{In this subsection we will recall the classification result of singular fibers of elliptic surfaces due to Kodaira~\cite{Kod}. 
Let $C$ be a smooth algebraic curve over $\C$, $X$ be a smooth algebraic surface over $\C$ and 
$$
  f\colon X\to C 
$$
a proper algebraic morphism whose generic fibers are smooth elliptic curves and whose fibers contain no rational curves of self-intersection $-1$. 
Let $c_0 \in C$ be arbitrary. Then, the fiber $X_{c_0} = f^{-1}(c_0 )$ is bianalytically isomorphic to a curve from the following list:
\begin{itemize}
 \item $I_{n+1} = A_n^{(1)}$, consisting of a cycle of $n+1$ rational curves of self-intersection $-2$ transversely intersecting each other for some $n\in \N$; 
 \item $I_{n-4}^* = D_n^{(1)}$, consisting of $n+1$ rational curves of self-intersection $-2$ transversely intersecting each other according to the graph in Figure~\ref{fig:sing} (a) for some $n\geq 4$;
 \item $E_6^{(1)}, E_7^{(1)}, E_8^{(1)}$, consisting of $7,8$ respectively $9$ rational curves of self-intersection $-2$ transversely intersecting each other according to the graphs in Figure~\ref{fig:sing} (d), (c) and (b); 
 \item $II$, a curve of geometric genus $0$ and self-intersection $0$ with a single cuspidal singular point (locally modelled by $y^2 = x^3$); 
 \item $III$, the union of two smooth rational curves of self-intersection $-2$ tangent to each other to order exactly $2$ at a point; 
 \item $IV$, the union of three smooth rational curves of self-intersection $-2$ transversely intersecting each other at a single point. 
\end{itemize}
The curve $I_0$ is by definition a smooth elliptic curve with some complex structure, and $I_0^* = D_4^{(1)}$ has a vertex of valency $4$, 
therefore the corresponding curves admit a $1$-dimensional modulus each (for $I_0^*$, the modulus is determined by the cross-ratio of the 
points of the central component where the four leaves intersect it). The curve $I_1$ has a single component of self-intersection $0$ with an ordinary double point.}

\section{Perverse Leray filtration}\label{sec:perverse}
We first deal with the left-hand side of \eqref{eq:PW}.
The irregular Hitchin map \eqref{eq:Hitchin_map} endows 
$$
  H^* (\Mod_{\Dol}^{PX}, \Q ) 
$$
with a finite increasing perverse filtration $P^{\bullet}$ through the perverse Leray spectral sequence. As usual, we set 
$$
  \Gr^P_k = P^k / P^{k-1}. 
$$

\begin{prop}\label{prop:Dol}
 We have 
 \begin{align*}
  \dim_{\Q} \Gr^P_{0} H^0 (\Mod_{\Dol}^{PX}, \Q ) & = 1 \\  
   \dim_{\Q} \Gr^P_{1} H^2 (\Mod_{\Dol}^{PX}, \Q ) & = d^{PX} \\
  \dim_{\Q} \Gr^P_{2} H^2 (\Mod_{\Dol}^{PX}, \Q ) & = 1 
\end{align*}
for some $d^{PX} \in \N$, and all the other graded pieces of $H^*$ for $P$ vanish. In particular, we have 
$$
  b_2 ( \Mod_{\Dol}^{PX} ) = 1 + d^{PX}
$$
and 
$$
  PH^{PX}(q,t) = 1 + d^{PX} q t^2 + q^2 t^2. 
$$
Furthermore, we have 
$$
  d^{PX} = 10 - \chi (F_{\infty}^{PX}),
$$
where $F_{\infty}^{PX}$ is the fiber at infinity of $\Mod_{\Dol}^{PX}$ listed in Table \ref{table:Dol}. 
\end{prop}

The specific forms of $PH^{PX}(q,t)$ can then easily be determined using Proposition \ref{prop:Dol} and the fibers $F_{\infty}^{PX}$, and for convenience they are included in Table \ref{table:Dol}. 

\begin{table}
\begin{center}
\begin{tabular}{|l|l|r|}
 \hline 
 $X$ & $F_{\infty}^{PX}$ & $PH^{PX}(q,t)$ \\
 \hline 
 \hline 
 $VI$ & $D_4^{(1)}$ & $1 + 4 q t^2 + q^2 t^2$ \\
 \hline 
 $V$ & $D_5^{(1)}$ & $1 + 3 q t^2 + q^2 t^2$ \\ 
 \hline 
 $V_{\degen}$ & $D_6^{(1)}$ & $1 + 2 q t^2 + q^2 t^2$ \\
 \hline 
 $III(D6)$ & $D_6^{(1)}$ & $1 + 2 q t^2 + q^2 t^2$ \\ 
 \hline 
 $III(D7)$ & $D_7^{(1)}$ & $1 +   q t^2 + q^2 t^2$ \\
 \hline 
 $III(D8)$ & $D_8^{(1)}$ & $1 +                q^2 t^2$ \\ 
 \hline 
 $IV$ & $E_6^{(1)}$ & $1 + 2 q t^2 + q^2 t^2$ \\
 \hline 
 $II$ & $E_7^{(1)}$ & $1 +   q t^2 + q^2 t^2$ \\ 
 \hline 
 $I$ & $E_8^{(1)}$ & $1 +                q^2 t^2$ \\
 \hline 
\end{tabular}
\end{center}
 \caption{Fiber at infinity and perverse Hodge polynomial of $\Mod_{\Dol}^{PX}$} 
 \label{table:Dol}
\end{table}

\begin{proof}
As $\Mod_{\Dol}^{PX}$ is a non-compact oriented $4$-manifold, by Poincar\'e duality we have 
$$
  b_4 ( \Mod_{\Dol}^{PX} ) = 0. 
$$
Let $\H$ denote hypercohomology of a complex of sheaves and $H$ stand for cohomology of a single sheaf. 
Here and throughout this section, for ease of notation we drop the subscript ${\Dol}$ and the superscript $PX$ of $\Mod_{\Dol}^{PX}$ whenever this latter is in subscript.  
It is known that there exists a spectral sequence ${}_L^{\mathfrak{p}} E_r^{k,l}$ called {\it perverse Leray spectral sequence} degenerating at the second page 
$$
  {}_L^{\mathfrak{p}} E_2^{k,l} =  {}^{\mathfrak{p}} H^k (Y, {}^{\mathfrak{p}} R^l h_* \underline{\Q}_{\Mod}) \Rightarrow H^{k+l} (\Mod_{\Dol}^{PX}, \Q). 
$$
\blue{The perverse filtration on $H^{\bullet} (\Mod_{\Dol}^{PX}, \Q)$ is defined as the filtration 
\begin{equation}\label{eq:perverse_filtration}
  P^p \H^{\bullet} (Y, \R h_* \underline{\Q}_{\Mod} ) = \oplus_{l\leq p} {}_L^{\mathfrak{p}} E_2^{k,l} 
\end{equation}
induced by $h$ \eqref{eq:Hitchin_map} on the terms of ${}_L^{\mathfrak{p}} E_2$.} 

\blue{In order to be able to compute the graded pieces of~\eqref{eq:perverse_filtration}, we first need a digression on the structure of the Dolbeault spaces. 
Namely, we know from \cite[Theorems~1.1,~1.2, Proposition~4.2]{ISS1}, \cite[Theorems~2.1,~2.2,~2.3]{ISS2} that for each 
$$
 X \in \{ I, III(D6), III(D7), III(D8) \} 
$$
there exists an embedding 
$$
  \Mod_{\Dol}^{PX} \hookrightarrow E(1) 
$$
of $\Mod_{\Dol}^{PX}$ into the rational elliptic surface 
\begin{equation}\label{eq:elliptic_surface}
   E(1) = \CP2 \# 9 \overline{\CP{}}^2 
\end{equation}
so that 
$$
  E(1) \setminus \Mod_{\Dol}^{PX} = F_{\infty}^{PX} 
$$
for some non-reduced curve $F_{\infty}^{PX}$, moreover there exists an elliptic fibration 
$$
  \tilde{h}: E(1) \to \CP1
$$
so that the following diagram commutes 
\begin{equation}\label{eq:diagram}
  \xymatrix{
    \Mod_{\Dol}^{PX} \ar[d]_h \ar@{^{(}->}[r] & E(1) \ar[d]^{\tilde{h}} \\
  \C \ar@{^{(}->}[r] & \CP1. }
\end{equation}
In particular, we have 
$$
  \tilde{h}^{-1} (\infty ) = F_{\infty}^{PX}. 
$$
The type of the curves $F_{\infty}^{PX}$ is determined by $X$ and is listed in Table \ref{table:Dol}. 
In addition, there is a map 
\begin{equation}\label{eq:ruling}
  p: E(1) \to \CP1 
\end{equation}
which is the $8$-point blow-up of the ruling of the Hirzebruch surface of degree $2$. 
In this picture, the central component of $F_{\infty}^{PX}$ is the fiber at infinity of $p$ and its other components arise from the proper 
transform of the blown-up fibers and possibly some exceptional divisors.  
}

\blue{
If the residues of the Higgs field at the simple poles are assumed to have distinct eigenvalues then exactly the same results hold in the cases $II, IV, V_{\degen}, V, VI$ too. 
For $X = II, IV$ this follows from \cite[Theorems~2.4,~2.6]{ISS2}, for $X =V_{\degen}, V, VI$ see \cite[Sections~1.3,~1.4]{ISS3}. 
On the other hand, in all cases $II, IV, V_{\degen}, V, VI$ where one of the poles is simple and the residue has equal eigenvalues, 
the same statement holds for the completed Dolbeault moduli spaces. 
In case $X = VI$ this is precisely \cite[Proposition~2.9]{ISS3}. 
\begin{lem}\label{lem:embedding}
In cases $X = II, IV, V_{\degen}, V$, assume that there is a simple pole of the Higgs field such that the residue has equal eigenvalues. 
Then, there exists an embedding 
$$
  \Mod_{\Dol}^{PX} \hookrightarrow E(1) 
$$
of the completed Dolbeault moduli spaces so that 
$$
  E(1) \setminus \Mod_{\Dol}^{PX} = F_{\infty}^{PX} 
$$
for the non-reduced curve $F_{\infty}^{PX}$ listed in Table \ref{table:Dol}, and the diagram \eqref{eq:diagram} commutes. 
\end{lem}
\begin{proof}
The proof is similar to the $PVI$ case. 
We consider the pencil of spectral curves associated to Higgs bundles with poles of the given local forms. 
According to \cite[Theorem~1.1]{Sz_BNR}, the moduli space arises as a certain relative compactified Picard-scheme of this pencil. 
In order to determine the relative compactified Picard-scheme, one first needs to blow up the base locus of the pencil of spectral curves; 
in general this process involves blowing up infinitesimally close points. 
The common phenomenon in the cases when the residue of the Higgs field at a simple pole $p_1$ has equal eigenvalues is that one exceptional divisor 
$E$ of the blow-up process (with self-intersection number equal to $(-2)$) maps to $p_1$ under the ruling and becomes a component of one of the fibers $X_t$ 
in the fibration. In the cases $X = II, IV$ this is precisely proven in \cite[Lemma~4.5]{ISS2}. 
The same proof goes along for the other types too, because both the assumptions and the assertion is local at the fiber of the ruling over $p_1$. 
Let us denote by $Z_t$ the singular curve in the pencil of spectral curves whose proper transform contains $E$, so that $X_t$ is the proper transform of $Z_t$. 
It follows from \cite[Section~6]{SSS} that $X_t$ is one of the Kodaira types 
$$
  I_2, I_3, I_4, III, IV. 
$$
The corresponding spectral curves $Z_t$ are listed before \cite[Lemma~10.1]{ISS2}, except in case $X_t$ is of type $I_4$. 
The case $I_4$ may only occur in cases $X = V_{\degen}, V, VI$, under the assumption that there exists two simple poles $p_1,p_2$ of the Higgs field such that for $i\in \{ 1,2 \}$ 
$\Res_{p_i}(\theta )$ has two equal eigenvalues. In this case two non-neighbouring components of $X_t$ get mapped to $p_1,p_2$ respectively under the ruling and 
$Z_t$ consists of two rational curves (sections of the Hirzebruch surface of degree $2$) intersecting each other transversely in two points, one on the fiber over $p_1$ and another one on the fiber over $p_2$. 
Indeed, $Z_t$ may have at most two components because it is a $2:1$ ramified covering of the base curve $\CP1$, so two components of $X_t$ must be exceptional divisors of the 
blow-up process; one of these two components must come from blow-ups at $p_1$ and the other one from blow-ups at $p_2$, for otherwise the dual graph could not be a cycle. 
By \cite[Lemma~10.1]{ISS2}, Higgs bundles whose residue at $p_1$ (and $p_2$ in case $IV$) has non-trivial nilpotent part correspond to locally free spectral sheaves over $Z_t$ at $p_1$ (respectively, $p_2$). 
For such curves $Z_t$, \cite[Lemma~10.2]{ISS2} determines the families of locally free spectral sheaves giving rise to parabolically stable Higgs bundles. 
On the other hand, any torsion-free sheaf on $Z_t$ is the direct image of a locally free sheaf on a partial normalization. 
Let us separate cases according to the type of $X_t$. 
\begin{enumerate}
 \item 
If $X_t$ is of type $I_2$ then $Z_t$ is a nodal rational curve with a single node on the fiber over $p_1$; there exists a family parameterized by $\C^{\times}$ of locally free sheaves 
and a unique torsion-free but not locally free sheaf of given degree on $Z_t$. 
This latter non-locally free torsion free sheaf gives rise to a unique Higgs bundle whose residue has the required eigenvalue of multiplicity $2$ and trivial nilpotent part. 
This object is irreducible hence stable. 
On the other hand, the choice of quasi-parabolic structure at $p_1$ compatible with this unique Higgs bundle is an arbitrary element of $\CP1$. 
This gives us that the Grothendieck class of the Hitchin fiber of the completed moduli space over the point $t$ is 
$$
  [\C^{\times}] + [\CP1 ]. 
$$
As the unique Kodaira fiber in this class is $I_2$, we deduce from Lemma~\ref{lem:smoothness} that the Hitchin fiber of the completed moduli space over the point $t$ is of this type, 
i.e. the same type as $X_t$. 
\item If $X_t$ is of type $I_3$ then $Z_t$ is composed of two sections of the ruling intersecting each other transversely in two distinct points, one of them lying on the fiber over $p_1$. 
As shown in \cite[Lemma~10.2.(2)]{ISS2}, Higgs bundles with spectral curve $Z_t$ and residue having non-trivial nilpotent part form a family parameterized by a variety in the class 
$$
  2 [\C^{\times}] + [pt]. 
$$
As in the previous point, there exists a single torsion-free but not locally free sheaf giving rise to a Higgs bundle with spectral curve $Z_t$ such that $\Res_{p_1}(\theta )$ 
has trivial nilpotent part. Again, the quasi-parabolic structure at $p_1$ compatible with this unique Higgs bundle is parameterized by $\CP1$, so the class of the 
Hitchin fiber of the completed moduli space over the corresponding point $t$ is 
$$
    2 [\C^{\times}] + [pt] + [\CP1 ]. 
$$
The only Kodaira fiber in this class is $I_3$, hence the Hitchin fiber of the completed moduli space over $t$ is $I_3$. 
\item For $X_t$ is of type $I_4$, as we already mentioned, $Z_t$ is a union of two sections of the ruling that intersect each other transversely in two points: 
one on the fiber over each of $p_1,p_2$. Therefore, the analysis is quite similar to the case of $I_3$ treated above:  Higgs bundles with spectral curve $Z_t$ such that 
both $\Res_{p_1}(\theta ), \Res_{p_2}(\theta )$ have non-trivial nilpotent part are parameterized by a variety in class 
$$
  2 [\C^{\times}]. 
$$
(The class of the point $[pt]$ that appears in the case $I_3$ is missing here because it corresponds to a torsion-free but non-locally free sheaf at $p_2$ 
which would give rise to a Higgs bundle with trivial nilpotent part at $p_2$.) 
Now, there exists a single sheaf of given degree that is locally free at $p_1$ and torsion-free but non-locally free at $p_2$. 
For the Higgs bundle obtained as the direct image of this sheaf, compatible quasi-parabolic structures at $p_1$ are parameterized by $\CP1$. 
The same observations clearly apply with $p_1, p_2$ interchanged too. Finally, notice that stability excludes that the spectral sheaf be 
torsion-free but non-locally free at both $p_1, p_2$: this would mean that the spectral sheaf comes from the normalization of $Z_t$, so 
the corresponding Higgs bundle would be decomposable. In sum, the class of the Hitchin fiber of the completed moduli space over $t$ is 
$$
  2 [\C^{\times}] + 2 [\CP1 ]. 
$$
As the only Kodaira fiber in this class is $I_4$, we infer that the Hitchin fiber of the completed moduli space over $t$ is of type $I_4$. 
\item If $X_t$ is of type $III$ then $Z_t$ is a cuspidal rational curve with a single cusp on the fiber over $p_1$. 
 Stability of any Higgs bundle with spectral sheaf supported on $Z_t$ again follows from irreducibility. 
 By virtue of \cite[Lemma~7.2]{ISS2}, locally free sheaves of given degree on $Z_t$ are parameterized by $\C^{\times}$
 and there exists a single non-locally free torsion free sheaf of given degree on $Z_t$. This latter gives rise to 
 a unique Higgs bundle in the extended moduli space with residue having trivial nilpotent part. 
 Again, compatible quasi-parabolic structures at $p_1$ provide a further $\CP1$ of parameters, so that  
 the class of the Hitchin fiber of the completed moduli space over $t$ is 
 $$
  [\C] + [\CP1 ]. 
 $$
 As the unique Kodaira fiber in this class is $III$, we see that the Hitchin fiber of the completed moduli space over $t$ is of type $III$. 
 \item If $X_t$ is of type $IV$ then $Z_t$ consists of two sections of the Hirzebruch surface, simply tangent to each other on the fiber over $p_1$.
  \cite[Lemma~10.2.(4)]{ISS2} implies that stable Higgs bundles with spectral curve $Z_t$ and $\Res_{p_1}(\theta )$ having non-trivial nilpotent part are 
  parameterized by a variety in class 
$$
  2 [\C], 
$$
while there exists a unique equivalence class of stable Higgs bundles with spectral curve $Z_t$ and $\Res_{p_1}(\theta )$ having trivial nilpotent part. 
 For this latter, we again have to add the parameter space $\CP1$ of compatible quasi-parabolic structures at $p_1$ to the above family, 
 leading to the class 
 $$
  2 [\C] + [\CP1 ].
 $$
 As the only Kodaira fiber in this class is $IV$, the corresponding Hitchin fiber is of type $IV$ too. 
\end{enumerate}
\end{proof}
\begin{remark}
In the Lemma we found that the completed Hitchin system has the same type of singular fibers as the associated fibration of spectral curves. 
A similar statement is shown in~\cite[Corollary~6.7]{AF}, based on the analysis~\cite{HRLT} of Fourier--Mukai transform for sheaves on 
various singular elliptic curves. Our result is more general than the one of~\cite{AF} in that it also treats the ramified 
Dolbeault moduli spaces and consequently more types of singular fibers enter into the picture, and we also consider the dependence 
of our result on parabolic weights. As $\mbox{Gl}(2, \C)$ is Langlands-selfdual, the above relative self-duality result can be considered 
as an irregular version of Mirror Symmetry of Hitchin systems~\cite{HT}. 
\end{remark}
}
\blue{We are now ready to return to the study of the perverse filtration~\eqref{eq:perverse_filtration}. 
Our computation closely follows~\cite[Proposition~4.11]{HST}. 
We first observe that since the fibers of $h$ are connected, ${}^{\mathfrak{p}} R^0 h_* \underline{\C}_{\Mod} = \underline{\C}_Y [-1]$, the trivial local system of rank $1$ over $Y$ placed in degree $1$. 
By the relative hard Lefschetz theorem, the same then holds for ${}^{\mathfrak{p}} R^2 h_* \underline{\C}_{\Mod}$ too.}
\blue{
We get that ${}_L^{\mathfrak{p}} E_2^{k,l}$ is of the form 
$$
  \xymatrix{
  k = 2 & 0 & \C^{b_3(\Mod)} & 0  \\
  k = 1 & 0 & {}^{\mathfrak{p}} H^1 (Y, {}^{\mathfrak{p}} R^1 h_* \underline{\C}_{\Mod})& 0 \\ 
  k = 0 & \C & \C^{b_1(\Mod)} & \C \\
    & l = 0 & l = 1 & l = 2 
  }
$$
The perverse Leray spectral sequence degenerates at this term. In particular, the dimension $b_1(\Mod)$ of ${}_L^{\mathfrak{p}} E_2^{0,1}$ is equal to the first Betti number $b_1(\Mod_{\Dol}^{PX} )$ 
and the dimension $b_3(\Mod)$ of ${}_L^{\mathfrak{p}} E_2^{2,1}$ is equal to the third Betti number $b_3(\Mod_{\Dol}^{PX} )$.} 
\begin{lem}\label{lem:b_1}
 We have $b_1(\Mod_{\Dol}^{PX} ) = 0$ and $b_3(\Mod_{\Dol}^{PX} ) = 0$. 
\end{lem}
\begin{proof}
Let $N$ denote a tubular neighbourhood of $F_{\infty}^{PX}$ in $E(1)$ and consider the covering 
$$
  E(1) = \Mod_{\Dol}^{PX} \cup N. 
$$
Part of the associated Mayer--Vietoris cohomology long exact sequence reads as 
\begin{align}
 \to H^1(E(1), \Q ) & \to H^1 (\Mod_{\Dol}^{PX} , \Q ) \oplus H^1 (N, \Q ) \to H^1 (\Mod_{\Dol}^{PX} \cap N, \Q  ) \xrightarrow{\delta} \label{eq:MV} \\
 \to H^2(E(1), \Q ) & \to \cdots \notag 
\end{align}
We know that $H^1(E(1), \Q)$ vanishes because it has the structure of a CW-complex only admitting even-dimensional cells.
On the other hand, $\Mod_{\Dol}^{PX} \cap N$ is homotopy equivalent to a plumbed $3$-manifold $Y_{\Dol}^{PX}$ over $F_{\infty}^{PX}$: 
\begin{equation}\label{eq:Dolbeault_plumbed}
   \Mod_{\Dol}^{PX} \cap N \cong Y_{\Dol}^{PX}. 
\end{equation}
To prove the assertion it is clearly sufficient to show that the connecting morphism $\delta$ between cohomology groups is a monomorphism, 
or dually, that the connecting morphism 
$$
  \partial: H_2 ( E(1), \Q ) \to H_1 ( \Mod_{\Dol}^{PX} \cap N , \Q )
$$
on singular homology is an epimorphism. 
According to Table \ref{table:Dol}, for each $X$ the Dynkin diagram of $F_{\infty}^{PX}$ is simply connected. 
As $Y_{\Dol}^{PX}$ is a Seifert fibered $3$-manifold, by~\cite{O} it is known that $H_1 ( \Mod_{\Dol}^{PX} \cap N , \Q )$ is generated by classes $[\gamma_i]$ of normal 
loops around the irreducible components $D_i$ of $F_{\infty}^{PX}$ corresponding to the central node and the leaves of the plumbing tree. 
Let us now recall the definition of $\partial$. Assume given a singular $2$-cycle $C$ in $E(1)$, that decomposes as 
\begin{equation}\label{eq:decomposition}
   C = A + B
\end{equation}
where $A$ and $B$ are singular $2$-chains in $\Mod_{\Dol}^{PX}$ and $N$ respectively. (Such a decomposition always exists using barycentric decomposition.) 
We then let 
$$
  \partial ([C]) = [\partial (A)] = - [\partial (B)]. 
$$
Let now $[\gamma_i]$ be a loop around any component $D_i$ of $F_{\infty}^{PX}$; it is sufficient to show that there exists a $2$-cycle $C_i$ such that 
$$
  \partial ([C_i]) = [\gamma_i]. 
$$
There are two cases to consider. First, $D_i$ may be the component corresponding to central node of $F_{\infty}^{PX}$; 
in this case, the cycle $C_i$ may be chosen as the general fiber of the ruling~\eqref{eq:ruling}. 
Second, $D_i$ may be the proper transform of a fiber of~\eqref{eq:ruling} or an exceptional divisor corresponding to a leaf of $F_{\infty}^{PX}$: 
in this case, $C_i$ may be chosen as the exceptional divisor of the blow-up of one of the points of $D_i$. 
This finishes the proof of vanishing of $b_1$. 

\blue{The assertion for $b_3$ again follows from the corresponding segment of Mayer--Vietoris sequence 
\begin{align*}
 \to H^3(E(1), \Q ) & \to H^3 (\Mod_{\Dol}^{PX} , \Q ) \oplus H^3 (N, \Q ) \to H^3 (\Mod_{\Dol}^{PX} \cap N, \Q  ) \xrightarrow{\delta} \\
 \to H^4(E(1), \Q ) & \to 0.
\end{align*}
Indeed, here $H^3(E(1), \Q ) = 0$ because $E(1)$ is the body of a CW-complex with only even-dimensional cells and  $H^4(E(1), \Q ) \cong \Q$ because $E(1)$ is a smooth oriented connected compact $4$-manifold. 
Moreover, $H^3 (\Mod_{\Dol}^{PX} \cap N, \Q  ) \cong \Q$ because  $\Mod_{\Dol}^{PX} \cap N$ has the smooth oriented compact connected $3$-manifold $Y_{\Dol}^{PX}$ as deformation retract. 
These isomorphisms then show that 
$$
  H^3 (\Mod_{\Dol}^{PX} , \Q ) \oplus H^3 (N, \Q ) \cong 0.
$$
}
\end{proof}
\blue{
In case the fibers of $h$ are not all integral, the usual Leray spectral sequence may have terms supported in dimension $0$. 
However, these extra terms occur in degrees $(2,1)$ and $(0,2)$ and get annihilated by the differential $\mbox{d}_2$. 
Hence, for the purpose of our study we may assume that the eigenvalues of $\theta$ are taken sufficiently generic, so 
that all singular fibers of $h$ are integral. In this case, the usual Leray spectral sequence degenerates at ${}_L E_2$ too. 
Since both ${}_L E_2$ and ${}_L^{\mathfrak{p}} E_2$ abut to $H^{\bullet} (\Mod_{\Dol}^{PX})$ and their terms of degrees $(0,2)$ and $(2,0)$ agree, we deduce the equality 
$$
  {}^{\mathfrak{p}} H^1 (Y, {}^{\mathfrak{p}} R^1 h_* \underline{\C}_{\Mod}) \cong H^1 (Y,  R^1 h_* \underline{\C}_{\Mod}). 
$$
Let us set 
$$
  d^{PX} = \dim_{\C} H^1 (Y, R^1 h_* \underline{\C}_{\Mod^{PX}}) 
$$
(where we have reintroduced the superscript $PX$ of $\Mod$ in order to emphasize the way in which the right-hand side depends on the specific irregular type that we work with). 
By the form of ${}_L^{\mathfrak{p}} E_2^{k,l}$ as showed above and the definition~\eqref{eq:perverse_filtration}, we see that 
\begin{align}
 \dim_{\Q} \Gr^P_{0} H^2 (\Mod_{\Dol}^{PX}, \Q) & = 1 \label{eq:GrP1}\\
 \dim_{\Q} \Gr^P_{1} H^2 (\Mod_{\Dol}^{PX}, \Q) & = d^{PX} \label{eq:GrP3} \\
 \dim_{\Q} \Gr^P_{2} H^2 (\Mod_{\Dol}^{PX}, \Q) & = 1. \label{eq:GrP2}
\end{align}
Moreover, we have 
\begin{equation}\label{eq:generatorP1}
 P^{1} H^2 (\Mod_{\Dol}^{PX}, \Q) = \Ker ( j^*  )
\end{equation}
where $Y_{-1} \in Y$ is a generic point in the Hitchin base and the morphism 
$$
  j^* : H^2 (\Mod_{\Dol}^{PX}, \Q) \to H^2 ( h^{-1} (Y_{-1}) , \Q) 
$$
is the induced morphism in cohomology by the inclusion 
$$
  j: h^{-1} (Y_{-1}) \to \Mod_{\Dol}^{PX}.
$$
}
Let us denote Euler-characteristic of a topogical space by $\chi$. 
\begin{lem}\label{lem:dPXF}
 We have 
 $$
  d^{PX} = 10 - \chi (F_{\infty}^{PX}). 
 $$
\end{lem}
\begin{proof}
We see from degeneration of the perverse Leray spectral sequence at ${}_L^{\mathfrak{p}} E_2$ that  
\begin{equation}\label{eq:dPXb_2}
   1 + d^{PX} = b_2 (\Mod_{\Dol}^{PX}).  
\end{equation} 
By Lemma \ref{lem:b_1} and additivity of $\chi$ with respect to stratifications we deduce 
$$
  b_0 (\Mod_{\Dol}^{PX}) + b_2 (\Mod_{\Dol}^{PX}) + \chi (F_{\infty}^{PX}) = \chi (E(1)) = 12. 
$$
The assertion follows because $\Mod_{\Dol}^{PX}$ is connected. 
\end{proof}
This Lemma and~\eqref{eq:GrP1},~\eqref{eq:GrP3}, \eqref{eq:GrP2} coupled with \eqref{eq:dPXb_2} finish the proof of the Proposition. 
\end{proof} 

\begin{remark}
 \blue{An alternative method of proof would be to make use of the relationship between the flag filtration and the perverse filtration~\cite[Theorem~4.1.1]{dCM}. 
 In our case however, the direct method works just as well. 
 The lattice $H^1 (Y, R^1 h_* \underline{\Z}_{\Mod})$ is isomorphic to the generic Mordell--Weil group $MW^{\sigma}(\tilde{h})$ 
 of the identity component $E(1)^{\sigma}$ (singled out by a fixed section $\sigma$) of a N\'eron model of $\tilde{h}\colon E(1) \to \CP1$. 
 Indeed, considering a cell decomposition of $\CP1$ such that no $1$-cell contains $\infty$ we see that 
 $$
 H^1 (Y, R^1 h_* \underline{\Z}_{\Mod}) \cong H^1 (\CP1 , R^1 h_* \underline{\Z}_{E(1)}). 
 $$ 
 Now, Kodaira's short exact sequence \cite[Section~V.9]{BHPV}
 $$
  0 \to R^1 h_* \underline{\Z}_{E(1)} \to R^1 h_* \O_{E(1)} \to \O_{\CP1} (E(1)^{\sigma}) \to 0
 $$
 gives rise to the cohomology long exact sequence 
 \begin{align*}
   \cdots \to & H^0 (\CP1 , R^1 h_* \O_{E(1)}) \to H^0 (\CP1 , \O_{\CP1} (E(1)^{\sigma}) ) \to  \\ 
   \to H^1 (\CP1 , R^1 h_* \underline{\Z}_{E(1)})\to  &  H^1 (\CP1 , R^1 h_* \O_{E(1)}). 
 \end{align*}
 In this sequence the two extremal terms vanish by relative duality \cite[Section~III.12]{BHPV}
 $$
 R^1 h_* \O_{E(1)} \cong (R^0 h_* \omega_{E(1)/ \CP1})^{\vee} \cong R^0 h_* K_{E(1)}^{\vee} \otimes K_{\CP1}, 
 $$
 the K\"unneth formula and rationality of $E(1)$. Therefore, for generic choices of the eigenvalues (so that the singular fibres of $\Mod_{\Dol}^{PX}$ are irreducible) 
 Proposition \ref{prop:Dol} also follows from the Shioda--Tate formula \cite{Sh1, Sh2, Tate}, which in our case reads 
 $$
  \mbox{rank}_{\Z} (MW^{\sigma}(\tilde{h})) = \mbox{rank}_{\Z} (NS(E(1))) - 1 - \mbox{rank}_{\Z} H^2(F_{\infty}^{PX}, \Z), 
 $$
 because the N\'eron--Severi group $NS(E(1))$ of $E(1)$ satisfies (again using rationality) 
 $$
  NS(E(1)) \cong H^2 (E(1), \Z) \cong \Z^{10}.
 $$
 }
\end{remark}

\section{Weight filtration}\label{sec:weight}
We now turn our attention to the right hand side of \eqref{eq:PW}. 
Observe first that according to \cite{PS}, for all $X$ the space $\Mod_{\Betti}^{PX}$ is a smooth affine cubic surface defined by a polynomial 
\begin{equation}\label{eq:cubic}
  f^{PX} (x_1, x_2, x_3) = x_1 x_2 x_3 + Q^{PX} (x_1, x_2, x_3) 
\end{equation}
for an affine quadric $Q^{PX}$. Each of these quadrics depends on some subset (possibly empty) of complex parameters 
\begin{equation}\label{eq:parameters}
   s_0, s_1, s_2, s_3, \alpha, \beta .
\end{equation}
For a generic choice of these parameters, that we will assume from now on, the obtained affine cubic surfaces are smooth. 
Moreover, in case the cubics do not depend on any parameter, the affine cubic surfaces are always smooth. 
Denote by 
$$
  F^{PX} \in \C [x_0, x_1, x_2, x_3]
$$
the homogenization of $f^{PX}$ as a homogeneous cubic polynomial and consider the projective surface 
\begin{equation}\label{eq:compactifcation_Betti}
 \overline{\Mod}_{\Betti}^{PX} = \mbox{Proj} ( \C [x_0, x_1, x_2, x_3] / (F^{PX}) ), 
\end{equation}
which is a compactification of ${\Mod}_{\Betti}^{PX}$. 
In general, $\overline{\Mod}_{\Betti}^{PX}$ is not smooth: it has some isolated singularities over $x_0 = 0$. 
Let us set 
\begin{equation}\label{eq:NPX}
   N^{PX} = \sum_{p} \mu (p )
\end{equation}
where $\mu$ stands for the Milnor number of an isolated surface singularity and the summation ranges over all singular points of $\overline{\Mod}_{\Betti}^{PX}$. 

\begin{prop}\label{prop:Betti}
 The non-trivial graded pieces of $H^*$ for $W$ are 
\begin{align*}
 \Gr^W_0 H^0 (  \Mod_{\Betti}^{PX} ) & \cong \C \\ 
 \Gr^W_{2} H^2 (\Mod_{\Betti}^{PX}, \C ) 
	& \cong  \C^{4 - N^{PX}} \\
 \Gr^W_{4} H^2 (\Mod_{\Betti}^{PX}, \C ) 
	  & \cong \C .
\end{align*}
In particular, we have 
$$
  b_2 ( \Mod_{\Betti}^{PX} ) = 5 - N^{PX}
$$
and 
$$
  WH^{PX}(q,t) = 1 + ( 4 - N^{PX} ) q t^2 + q^{2} t^2. 
$$
\end{prop}

The singularities of $\overline{\Mod}_{\Betti}^{PX}$ and the weight polynomial of ${\Mod}_{\Betti}^{PX}$ in the various cases are summarized in Table \ref{table:B}. 

\begin{table}
\begin{center}
\begin{tabular}{|l|l|r|}
 \hline 
 $X$ & Singularities of $\overline{\Mod}_{\Betti}^{PX}$ & $WH^{PX}(q,t)$ \\
 \hline 
 \hline 
 $VI$ & $\emptyset$ & $1 + 4 q t^2 + q^{2} t^2$ \\
 \hline 
 $V$ & $A_1$ & $1 + 3 q t^2 + q^{2} t^2$ \\ 
 \hline 
 $V_{\degen}$ & $A_2$ & $1 + 2 q t^2 + q^{2} t^2$ \\
 \hline 
 $III(D6)$ & $A_2$ & $1 + 2 q t^2 + q^{2} t^2$ \\ 
 \hline 
 $III(D7)$ & $A_3$ & $1 + q t^2 + q^{2} t^2$ \\
 \hline 
 $III(D8)$ 
					& $A_4$ & $1 +   q^{2} t^2$ \\ 
 \hline 
 $IV$ & $A_1 + A_1$ & $1 + 2 q t^2 + q^{2} t^2$ \\
 \hline 
 $II$ & $A_1 + A_1 + A_1$ & $1 + q t^2 + q^{2} t^2$ \\ 
 \hline 
 $I$ & $A_2 + A_1 + A_1$ & $1 + q^{2} t^2$ \\
 \hline 
\end{tabular}
\end{center}
 \caption{Singularities of $\overline{\Mod}_{\Betti}^{PX}$ and weight Hodge polynomial of ${\Mod}_{\Betti}^{PX}$}
 \label{table:B}
\end{table}

\begin{proof}
We use the definition given in \cite{DelHodge2} of the weight filtration on the mixed Hodge structure on the cohomology of an affine variety in terms of a smooth projective compactification. 
The form \eqref{eq:cubic} of $f^{PX}$ implies that the compactifying divisor 
$$
   D = \overline{\Mod}_{\Betti}^{PX} \setminus {\Mod}_{\Betti}^{PX} = (x_0) \cap (F^{PX}) \subset \CP{2}
$$
in $\overline{\Mod}_{\Betti}^{PX}$ is defined by the equation 
$$
  x_1 x_2 x_3, 
$$
so 
\begin{equation}\label{eq:D}
  D = L_1 \cup L_2 \cup L_3
\end{equation}
where each $L_i$ is a complex projective line such that each two of them $L_i, L_j$ for $i< j$ intersect each other transversely in a point $p_{ij}$. 
Said differently, the nerve complex of $D$ consists of the edges (and vertices) of a triangle \blue{$A_2^{(1)}$}. In particular, the body of this complex is homeomorphic to a circle $S^1$. 
As we will see in Subsections \ref{subsec:BVI}--\ref{subsec:BI}, all singularities of $\overline{\Mod}_{\Betti}^{PX}$ are located at some of the points $p_{ij}$ and are of type $A_k$ for $k = \mu(p_{ij} )$. 
We obtain a smooth compactification $\widetilde{\Mod}_{\Betti}^{PX}$ of $\Mod_{\Betti}^{PX}$ by taking the minimal resolution of the singularities $p_{ij}$ of $\overline{\Mod}_{\Betti}^{PX}$. 
It follows that there exists a smooth compactification $\widetilde{\Mod}_{\Betti}^{PX}$ of $\Mod_{\Betti}^{PX}$ by a normal crossing divisor 
\begin{equation}\label{eq:compactifying_divisor}
   D^{PX}
\end{equation}
consisting of reduced projective lines. We know that $D^{PX}$ contains the proper transform of each component of \eqref{eq:D}. 
\blue{More precisely, the nerve complex $\mathcal{N}^{PX}$ of $D^{PX}$ arises from the graph $A_2^{(1)}$ of \eqref{eq:D} by replacing the edge corresponding to an intersection point $p_{ij}$ 
by a diagram $A_{\mu(p_{ij} )}$. 
On the other hand, the generic plane section of $\widetilde{\Mod}_{\Betti}^{PX}$ is a cubic curve, therefore the nerve complex $ \mathcal{N}^{PX}$ must appear on Kodaira's list 
given in Subsection~\ref{ssec:elliptic} (up to modifying the self-intersection numbers by blow-ups necessary to eliminate the singular points). 
From this we see that the nerve complex of $D^{PX}$ is a cycle of length $N^{PX} + 3$ 
\begin{equation}\label{eq:nerve}
   \mathcal{N}^{PX} = C_{N^{PX} + 3}.
\end{equation}
For more details and the self-intersection numbers see~\cite[Lemma~1]{N-Sz}.}
As customary, we will denote by $\mathcal{N}^{PX}_0$ and $\mathcal{N}^{PX}_1$ the set of $0$- and $1$-dimensional cells of $\mathcal{N}^{PX}$, respectively. 

We are now ready to determine the Betti numbers of $\widetilde{\Mod}_{\Betti}^{PX}$. 

\begin{lem}\label{lem:b_2}
 We have 
 \begin{align*}
  b_0 \left( \widetilde{\Mod}_{\Betti}^{PX} \right) & = 1 = b_4 \left( \widetilde{\Mod}_{\Betti}^{PX} \right) \\ 
  b_1 \left( \widetilde{\Mod}_{\Betti}^{PX} \right) & = 0 = b_3 \left( \widetilde{\Mod}_{\Betti}^{PX} \right) \\
  b_2 \left( \widetilde{\Mod}_{\Betti}^{PX} \right) & = 7. 
 \end{align*}
\end{lem}
\begin{proof}
 The assertion for $b_0$ is obvious, and then immediately follows by Poincar\'e duality for $b_4$ too. 

 In case $N^{PX} = 0$, i.e. $\overline{\Mod}_{\Betti}^{PX}$ is a smooth projective cubic surface, it is known that $\overline{\Mod}_{\Betti}^{PX}$ is given by a blow-up of 
 $\CP2$ in six different points, and so carries the structure of a CW-complex with only even-dimensional cells, with $7$ two-dimensional cells. 

The non-smooth surfaces $\overline{\Mod}_{\Betti}^{PX}$ clearly belong to the $20$-dimensional family of projective cubic surfaces. 
The points parameterizing smooth cubics form a dense set in $\C^{20}$ with respect to the analytic topology. 
We will see in Subsections \ref{subsec:BVI}--\ref{subsec:BI} that the spaces $\overline{\Mod}_{\Betti}^{PX}$ only admit singularities of type $A_k$. 
It is known that a smoothing of a projective surface with $ADE$ singularities coincides up to diffeomorphism with a minimal resolution thereof. 
In our case, a smoothing is a smooth cubic surface. The smooth case treated in the previous paragraph therefore implies the general statement. 
\end{proof}

Now, \cite[Th\'eor\`eme~$(3.2.5)$]{DelHodge2} implies that there exists a spectral sequence 
$$
    {}_W E_1^{-n,k+n} = H^{k-n} (\tilde{Y}^n) \Rightarrow H^k (\Mod_{\Betti}^{PX}, \C )
$$
endowed with the weight filtration, with first page ${}_W E_1$ of the form 
$$
  \xymatrix{
    k + n = 4 & \oplus_{p \in \mathcal{N}^{PX}_1} H^0 (p, \C ) \ar[r]^{\delta} & \oplus_{L \in \mathcal{N}^{PX}_0} H^2 (L, \C ) \ar[r]^{\delta_4} & H^4 \left( \widetilde{\Mod}_{\Betti}^{PX}, \C \right) \\
    k + n = 3 & 0 & \oplus_{L \in \mathcal{N}^{PX}_0} H^1 (L, \C ) \ar[r]^{\delta_3} &  H^3 \left( \widetilde{\Mod}_{\Betti}^{PX}, \C \right)  \\
    k + n = 2 & 0 & \oplus_{L \in \mathcal{N}^{PX}_0} H^0 (L, \C ) \ar[r]^{\delta_2} &  H^2 \left( \widetilde{\Mod}_{\Betti}^{PX}, \C \right)  \\
    k + n =   1 & 0 & 0 & H^1 \left( \widetilde{\Mod}_{\Betti}^{PX}, \C \right)  \\
    k + n =   0 & 0 & 0 & H^0 \left( \widetilde{\Mod}_{\Betti}^{PX}, \C \right) \\
    & -n= -2 & -n= -1 & -n= 0
    }
$$
Let us list a few properties (either obvious or directly following from \cite[Th\'eor\`eme~$(3.2.5)$]{DelHodge2}) related to this spectral sequence. 
\begin{enumerate}
 \item In the notation of \cite{DelHodge2}, we have $\tilde{Y}^1 = \coprod_{L \in \mathcal{N}^{PX}_0} L$ and $\tilde{Y}^2 = \coprod_{p \in \mathcal{N}^{PX}_1} p$. 
 \item The sequence degenerates at ${}_W E_2$. 
 \item The filtration $W_N$ is induced by $n \leq N$ on the above diagram, and its shifted filtration $W[k]$ defines the mixed Hodge structure on $H^k (\Mod_{\Betti}^{PX}, \C )$. 
 \item Up to identifying $H^2 (L, \C )$ with $H^0 (L, \C )$ via the Lefschetz operator, the map $\delta$ is the differential of the simplicial complex $\mathcal{N}^{PX}$. 
  In particular, as the body of $\mathcal{N}^{PX}$ is homeomorphic to $S^1$, we have 
  \begin{equation}\label{eq:ker(delta)}
    \dim_{\C} \Ker (\delta ) = 1 = \dim_{\C} \Coker (\delta).
  \end{equation}
 \item The morphisms $\delta_2, \delta_4$ are induced by the Thom morphism of the normal bundles of the subvarieties $L \hookrightarrow \widetilde{\Mod}_{\Betti}^{PX}$. 
 \item We have ${}_W E_1^{-1,3} = 0$ because the components $L$ are $2$-spheres, in particular simply connected. \label{property:E_1^(-1,3)=0}
 \item Lemma \ref{lem:b_2} shows that we have ${}_W E_1^{0,3} = 0 = {}_W E_1^{0,1}$, so the entire rows $k+n = 1$ and $k+n = 3$ are $0$. 
\end{enumerate}

\begin{lem}\label{lem:delta4}
 The morphism $\delta_4$ is an epimorphism. 
\end{lem}
\begin{proof}
Assuming that $\delta_4$ vanish, the term $H^4 \left( \widetilde{\Mod}_{\Betti}^{PX}, \C \right)$ in the spectral sequence could not be annihilated at any further page by any other term, so we 
would get 
$$
  H^4 \left( {\Mod}_{\Betti}^{PX}, \C \right) \neq 0. 
$$
This, however, would contradict that ${\Mod}_{\Betti}^{PX}$ is an oriented, non-compact $4$-manifold. 

It is also easy to derive the result directly using the explicit description of $\delta_4$ as wedge product by the Thom class $\Phi_L$ of the tubular neighbourhood $N_L$ 
of $L$ in $\widetilde{\Mod}_{\Betti}^{PX}$, as in Lemma \ref{lem:delta_2} below. 
Indeed, the image of the class of a generator $[\omega_L]$ of $H^2(L, \C)$ is then represented in $N_L$ by the compactly supported $4$-form 
$$
  \delta_4 ( [\omega_L] ) =  [\omega_L \wedge \Phi_L].
$$
The normal bundle of $L$ in $\widetilde{\Mod}_{\Betti}^{PX}$ is orientable, and the above $4$-form is cohomologous to a positive multiple of a volume form of $N_L$. 
This implies the assertion. 
\end{proof}

The Lemma and \eqref{eq:ker(delta)} now imply that in the top row of ${}_W E_2$ the only non-vanishing term will be the upper-left entry, and it is of dimension $1$.  

\begin{lem} \label{lem:delta_2}
The morphism $\delta_2$ is a monomorphism. 
\end{lem}
\begin{proof}
Consider a tubular neighbourhood $N_L$ of $L$ in $\widetilde{\Mod}_{\Betti}^{PX}$, 
diffeomorphic to the normal bundle of $L$ in $\widetilde{\Mod}_{\Betti}^{PX}$. 
We have the Thom morphism
\begin{align*}
   \iota_{L!} : H^0 (L, \C ) & \to H^2 \left( N_L, \C \right) \\
   1 & \mapsto \Phi_L
\end{align*}
where 
$$
  \iota_L: L \to N_L
$$
is the inclusion map and $\Phi_L$ stands for the Thom class of $L$ in $N_L$. 
Let now 
$$
  j_L : N_L \to \widetilde{\Mod}_{\Betti}^{PX}
$$
be the inclusion map. 
Notice that as $\Phi_L$ is vertically of compact support, its class can be extended by $0$ to define a class 
\begin{equation*}
  j_{L!} \Phi_L \in H^2\left(\widetilde{\Mod}_{\Betti}^{PX}, \C \right). 
\end{equation*}
The restriction of $\delta_2$ to the component $H^0 (L, \C )$ then maps 
$$
 H^0 (L, \C ) \ni 1 \mapsto j_{L!} \Phi_L \in H^2\left(\widetilde{\Mod}_{\Betti}^{PX}, \C \right).
$$
Now, according to Proposition 6.24 \cite{BT}, we have 
$$
  j_{L!} \Phi_L = PD_{N_L} ( [L] ) = PD_{\widetilde{\Mod}_{\Betti}^{PX}} ( [L] ), 
$$
where $PD_V$ stands for Poincar\'e duality in $V$ and $[L]$ is the cohomology class defined by integration on $L$. Therefore, for any 
$$
  (n_L)_{L \in \mathcal{N}^{PX}_0} \in \bigoplus_{L \in \mathcal{N}^{PX}_0} H^0 (L, \C )
$$
we have 
$$
  \delta_2 ( (n_L)_{L \in \mathcal{N}^{PX}_0} ) = PD_{\widetilde{\Mod}_{\Betti}^{PX}} \left( \sum_{L \in \mathcal{N}^{PX}_0} n_L [L] \right). 
$$
As Poincar\'e duality is perfect, the assertion is equivalent to showing that the classes $[L]$ for $L\in \mathcal{N}^{PX}_0$ are linearly independent in $H^2\left(\widetilde{\Mod}_{\Betti}^{PX}, \C \right)$. 

For this purpose, we fix a generic line $\ell$ in the projective plane $x_0 = 0$, and let 
$$
  \CP2_t, \quad t\in \CP1
$$
denote the pencil of projective planes in $\CP3$ passing through $\ell$. 
We may assume that $t = \infty$ corresponds to the plane $x_0 = 0$. For each $t \in \CP1$, the curve 
\begin{equation}\label{eq:elliptic_pencil}
   E_t = \CP2_t \cap \overline{\Mod}_{\Betti}^{PX}
\end{equation}
is an elliptic curve. The line $\ell$ intersects for each $k \in \{ 1,2,3 \}$ the line $L_k$ in a single point $p_k$, 
which is (by genericity of $\ell$) different from all the intersection points $p_{ij}$. 
The elliptic pencil \eqref{eq:elliptic_pencil} has base locus $B = \{ p_1 , p_2 , p_3 \}$. Let us consider the quadratic transformation 
\begin{equation}\label{eq:omega}
   \omega: E \to \widetilde{\Mod}_{\Betti}^{PX}
\end{equation}
with center $B$; $E$ is then an elliptic surface over $\CP1$, in particular it is diffeomorphic to \eqref{eq:elliptic_surface}.  
The exceptional divisors $\omega^{-1}(p_k)$ are sections of $Y$, in particular they do not belong to the fiber $E_{\infty}$ over $\infty$. 
Let us denote by $\tilde{L}$ the proper transform of $L$ with respect to $\omega$. We may then write 
$$
  \omega^* [L] = [\tilde{L}] + \sum_k m_{L,k} [\omega^{-1}(p_k)]
$$
for some $m_{L,k} \in \{ 0, 1 \}$. The quotient 
$$
  H^2\left(E, \C \right) / H^2\left(\widetilde{\Mod}_{\Betti}^{PX}, \C \right)
$$
is spanned by the classes $[\omega^{-1}(p_k)], k \in \{ 1,2,3 \}$. 
From this we see that if the classes $[L], L\in \mathcal{N}^{PX}_0$ were linearly dependent in $H^2\left(\widetilde{\Mod}_{\Betti}^{PX}, \C \right)$ 
then so would be the classes $[\tilde{L}], L\in \mathcal{N}^{PX}_0$ in $H^2\left(E, \C \right)$. 
Now, the fiber of the elliptic fibration $E$ over $t = \infty$ is equal to 
$$
  E_{\infty} \cong \widetilde{\Mod}_{\Betti}^{PX} \setminus {\Mod}_{\Betti}^{PX} = D^{PX} =  \cup_{L \in \mathcal{N}^{PX}_0} \tilde{L}. 
$$
\blue{We have already determined the type of $E_{\infty}$ in~\eqref{eq:nerve},}
in particular its intersection form is negative semi-definite, with non-trivial radical. 
The only possible vanishing linear combination of these classes would then be one in the radical of \eqref{eq:nerve}, generated by $(1, \ldots , 1)$. 
However, one sees immediately that the intersection number of 
\begin{equation}\label{eq:radical}
    \sum_{L \in \mathcal{N}^{PX}_0} [\tilde{L}] 
\end{equation}
and any $[\omega^{-1}(p_k)]$ is equal to $1$,  hence \eqref{eq:radical} is a non-zero class. 
Alternatively, 
the hyperplane class $[H]$ must intersect the class~\eqref{eq:radical} positively because the orthogonal complement of $[H]$ in $H^2\left(E, \C \right)$ is negative definite while the 
lattice of \eqref{eq:nerve} has non-trivial radical.
\end{proof}

The spectral sequence degenerates at ${}_W E_2$ and the weight with respect to the filtration $W[k]$ defining the mixed Hodge structure 
on $H^* ({\Mod}_{\Betti}^{PX}, \C)$ is defined by 
$$
  {}_W E_2^{-n,k+n} \mapsto k + n. 
$$
Taking into account Lemma \ref{lem:delta4} we derive that the only non-vanishing graded pieces of the weight filtration on cohomology read as: 
\begin{align*}
 \Gr^W_0 H^0 (  \Mod_{\Betti}^{PX} , \C ) & \cong \C \\ 
 \Gr^W_{2} H^2 (\Mod_{\Betti}^{PX}, \C ) & = \Coker \left( \delta_2:  \oplus_{L \in \mathcal{N}^{PX}_0} H^0 (L, \C ) \to H^2 \left( \widetilde{\Mod}_{\Betti}^{PX}, \C \right) \right) \\
 \Gr^W_{4} H^2 (\Mod_{\Betti}^{PX}, \C ) & = \Ker \left( \delta: \oplus_{p \in \mathcal{N}^{PX}_1} H^0 (p, \C ) \to \oplus_{L \in \mathcal{N}^{PX}_0} H^2 (L, \C )  \right) \\ 
      & = H^1 (S^1, \C)  \cong \C.
\end{align*} 
In particular, we have 
\begin{equation}\label{eq:generatorW2}
  W^2 H^2 (\Mod_{\Betti}^{PX}, \C ) = \im ( i^* ) 
\end{equation}
where 
$$
  i^* : H^2 \left( \widetilde{\Mod}_{\Betti}^{PX}, \C \right) \to H^2 (\Mod_{\Betti}^{PX}, \C )
$$
is the morphism induced by inclusion 
$$
  i: \Mod_{\Betti}^{PX} \to \widetilde{\Mod}_{\Betti}^{PX}. 
$$
Recalling \eqref{eq:nerve} that $\mathcal{N}^{PX}$ is a cycle of length $N^{PX} + 3$, Lemmas \ref{lem:b_2} and \ref{lem:delta_2} finish the proof \blue{of Proposition~\ref{prop:Betti}.}
\end{proof}

It remains to compute $N^{PX}$ and compare the perverse and weight polynomials explicitly for each $X$. 
We will determine $N^{PX}$ using the explicit form of the quadratic terms $Q^{PX}$ provided in \cite{PS}. 
Before turning to the study of the various cases, let us address a result that will be needed in some of the cases. 
Namely, assume that $\overline{\Mod}_{\Betti}^{PX}$ has a singularity at $[0:0:0:1]$. Plugging $x_3 = 1$ into $F^{PX}$ we get 
$$
  F^{PX} (x_0, x_1, x_2, 1) = f_2 (x_0, x_1, x_2) + f_3 (x_0, x_1, x_2)
$$
with $f_i$ homogeneous of order $i$. If $f_2$ is a non-degenerate quadratic form, then the Hessian of $F^{PX}$ at the singular point is  
non-degenerate, and so the singularity is of type $A_1$. Up to exchanging $x_0$ and $x_2$, we have the following. 
\begin{lem}\cite[Lemma~3$(c)$]{BW}\label{lem:A_k}
 Assume that $\overline{\Mod}_{\Betti}^{PX}$ has a singularity at $[0:0:0:1]$ and that with the above notations we have $f_2 = x_1 x_2$. 
 Then, if $f_3(1,0,0) \neq 0$ then $\overline{\Mod}_{\Betti}^{PX}$ has a singularity of type $A_2$ at $[0:0:0:1]$. 
 If $(1,0,0)$ is a $k_i$-tuple intersection of $x_i = 0$ with $f_3 = 0$, $i = 1,2$, then $[0:0:0:1]$ is an $A_{k_1 + k_2 + 1}$ singularity for 
 $$
  \{ k_1 , k_2 \} =  \{ 1 , 1 \}, \quad \{ 1 , 2 \}, \quad \{ 1 , 3 \}.
 $$
\end{lem}
The study of the specific cases, based on Lemma \ref{lem:A_k}, is contained in Subsections \ref{subsec:BVI}--\ref{subsec:BI} below. 
We note that in Subsections \ref{subsec:BVI}, \ref{subsec:BV}, \ref{subsec:BIV}, \ref{subsec:BIII(D6)} and \ref{subsec:BII} 
we rederive the weight polynomials obtained in Section 6 of \cite{HMW} using different methods.

\subsection{Case $X = VI$} \label{subsec:BVI}

In this case the quadric is of the form 
$$
  Q^{PVI} = x_1^2 + x_2^2 + x_3^2 - s_1 x_1 - s_2 x_2 - s_3 x_3 + s_4
$$
with $s_1, s_2, s_3, s_4 \in \C$. 
This is the generic quadric, so it is smooth at infinity, and 
$$
  WH^{PVI} (q,t) = 1 + 4 q t^2 + q^{2} t^2. 
$$

\subsection{Case $X = V$} \label{subsec:BV}

In this case the quadric is of the form 
$$
  Q^{PV} = x_1^2 + x_2^2 - (s_1 + s_2 s_3) x_1 - (s_2 + s_1 s_3) x_2 - s_3 x_3 + s_3^2 + s_1 s_2 s_3 +1 
$$
with $s_1, s_2 \in \C, s_3 \in \C^{\times}$. We have 
$$
  F^{PV} =  x_1 x_2 x_3 + x_0 x_1^2 + x_0 x_2^2 - (s_1 + s_2 s_3) x_0^2 x_1 - (s_2 + s_1 s_3) x_0^2 x_2 - s_3 x_0^2 x_3 + (s_3^2 + s_1 s_2 s_3 +1)  x_0^3. 
$$
An easy computation gives that the only singular point of $\overline{\Mod}_{\Betti}^{PV}$ over $x_0 = 0$ is $[0:0:0:1]$. 
We consider the affine chart $x_3 \neq 0$ and normalize $x_3 = 1$. 
Then, we have 
$$
  f_2 = x_1 x_2  - s_3 x_0^2, 
$$
which is a non-degenerate quadratic form because $s_3 \neq 0$. We infer that this singular point is of type $A_1$, in particular its Milnor number is $1$, 
hence 
$$
  WH^{PV} (q,t) = 1 + 3 q t^2 + q^{2} t^2. 
$$

\subsection{Case $X = V_{\degen}$} \label{subsec:BVdegen}

In this case the quadric is of the form 
$$
  Q^{PV_{\degen}} = x_1^2 + x_2^2 + s_0 x_1 + s_1 x_2 + 1
$$
with $s_0, s_1 \in \C$. 
The same analysis as in Subsection \ref{subsec:BV} shows that $[0:0:0:1]$ is the only singular point. 
This time, however, we have 
$$
 f_2 =  x_1 x_2, 
$$
which is degenerate. On the other hand, we have 
$$
  f_3 = x_0 x_1^2 + x_0 x_2^2 + s_0 x_0^2 x_1 + s_1 x_0^2 x_2 + x_0^3, 
$$
in particular $f_3(1,0,0) = 1$. Lemma \ref{lem:A_k} shows that the singularity is of type $A_2$, of $\mu = 2$, hence 
$$
    WH^{PV_{\degen}} (q,t) = 1 + 2 q t^2 + q^{2} t^2. 
$$

\subsection{Case $X = IV$} \label{subsec:BIV}

In this case the quadric is of the form 
$$
  Q^{PIV} = x_1^2 -(s_2^2 + s_1 s_2)x_1 - s_2^2 x_2 - s_2^2 x_3 + s_2^2 + s_1 s_2^3 
$$
with $s_1 \in \C, s_2 \in \C^{\times}$. 
We have 
$$
  F^{PIV} =  x_1 x_2 x_3 + x_0 x_1^2 -(s_2^2 + s_1 s_2) x_0^2 x_1 - s_2^2 x_0^2 x_2 - s_2^2 x_0^2 x_3 + (s_2^2 + s_1 s_2^3) x_0^3, 
$$
and the singular points of $\overline{\Mod}_{\Betti}^{PIV}$ over $x_0 = 0$ are $[0:0:1:0]$ and $[0:0:0:1]$. 
In the first point, the second-order homogeneous term of $F^{PIV}$ in affine co-ordinates $(x_0, x_1, x_3)$ is given by 
$$
  x_1 x_3 - s_2^2 x_0^2, 
$$
which is non-degenerate because $s_2 \neq 0$, so this singular point is of type $A_1$. 
In the second point, the second-order homogeneous term of $F^{PIV}$ in affine co-ordinates $(x_0, x_1, x_2)$ is given by 
$$
  f_2 = x_1 x_2 - s_2^2 x_0^2, 
$$
which shows that this singular point is again of type $A_1$. We infer that $\overline{\Mod}_{\Betti}^{PIV}$ has two singular points, 
each of Milnor number $1$, and 
$$
  WH^{PIV} (q,t) = 1 + 2 q t^2 + q^{2} t^2. 
$$

\subsection{Case $X = III(D6)$} \label{subsec:BIII(D6)}

In this case the quadric is of the form 
$$
  Q^{PIII(D6)} = x_1^2 + x_2^2 + (1 + \alpha \beta ) x_1 + (\alpha + \beta) x_2 + \alpha \beta
$$
with $\alpha, \beta \in \C^{\times}$.
The only singular point of $\overline{\Mod}_{\Betti}^{PIII(D6)}$ is $[0:0:0:1]$, with degree two term 
$$
  f_2 = x_1 x_2. 
$$
This time we have 
$$
  f_3 = x_0 x_1^2 + x_0 x_2^2 + (1 + \alpha \beta ) x_0^2 x_1 (\alpha + \beta) x_0^2 x_2 + \alpha \beta x_0^3. 
$$
Now, we again see that $f_3(1,0,0) = \alpha \beta \neq 0$, so Lemma \ref{lem:A_k} implies that we have an $A_2$-singularity, thus 
$$
    WH^{PIII(D6)} (q,t) = 1 + 2 q t^2 + q^{2} t^2. 
$$

\subsection{Case $X = III(D7)$} \label{subsec:BIII(D7)}
This case is obtained from degeneration of Subsection \ref{subsec:BIII(D6)} by setting the parameter $\beta$ of $Q^{PIII(D6)}$ 
(corresponding to the eigenvalue of the formal monodromy at one of the the irregular singular points) equal to $0$. 
In this case (up to exchanging the variables $x_1, x_2$) the quadric is of the form 
$$
  Q^{PIII(D7)} = x_1^2 + x_2^2 + \alpha x_1 + x_2
$$
with $\alpha \in \C^{\times}$. 
The only singular point of $\overline{\Mod}_{\Betti}^{PIII(D7)}$ is $[0:0:0:1]$, with homogeneous terms 
$$
  f_2 = x_1 x_2, \quad f_3 = x_0 x_1^2 + x_0 x_2^2 + \alpha x_0^2 x_1 + x_0^2 x_2. 
$$
This time $f_2$ is degenerate and we have $f_3 (1,0,0) = 0$, so the singularity is neither of type $A_1$ nor of type $A_2$. 
Plugging $x_1 = 0$ in $f_3$ gives 
$$
  f_3 (x_0, 0, x_2) = x_0 x_2^2 + x_0^2 x_2. 
$$
As this form has non-trivial linear term in $x_2$ at $x_0 = 1$, we get that $k_1 = 1$. 
Similarly, from 
$$
  f_3 (x_0, x_1, 0) = x_0 x_1^2 + \alpha x_0^2 x_1 
$$
and $\alpha \neq 0$ we deduce $k_2 = 1$. According to Lemma \ref{lem:A_k} the singular point is of type $A_3$, of $\mu = 3$, and we obtain 
$$
    WH^{PIII(D7)} (q,t) = 1 + q t^2 + q^{2} t^2. 
$$

\subsection{Case $X = III(D8)$} \label{subsec:BIII(D8)}

This case is obtained from further degeneration of Subsection \ref{subsec:BIII(D7)} by setting the parameter $\alpha$ of $Q^{PIII(D7)}$ 
(corresponding to the eigenvalue of the formal monodromy at the only remaining unramified irregular singularity) equal to $0$ too. 
We find the quadric\footnote{Notice that this differs from the result $x_1 x_2 x_3 + x_1^2 - x_2^2 - 1$ obtained in \cite[3.6]{PS}. 
We are grateful to Masa-Hiko Saito for pointing out that in this case the monodromy data has the extra symmetry $x_i \mapsto - x_i$ for $i\in \{ 1,2 \}$. 
Indeed, the two-fold Weyl group $\mathfrak{S}_2 \times \mathfrak{S}_2$ acts on the monodromy data by passing to opposite Borel subgroups at the 
two irregular singular points, and only the diagonal $\mathfrak{S}_2$ leaves invariant the constraints on the parameters. 
Now, introducing the invariant co-ordinates $y_1 = x_1^2, y_2 = x_2^2, y_3 = x_1 x_2$ and eliminating $y_2$ we are led to the 
formula $y_1 y_3 x_3 + y_1^2 -y_3^2 - y_1$, which in turn transforms into \eqref{eq:quadric_D8} after some obvious changes of co-ordinates.} 
\begin{equation}\label{eq:quadric_D8}
  Q^{PIII(D8)} = x_1^2 + x_2^2 + x_2. 
\end{equation}
The only singular point of $\overline{\Mod}_{\Betti}^{PIII(D8)}$ is $[0:0:0:1]$, with homogeneous terms 
$$
  f_2 = x_1 x_2, \quad f_3 = x_0 x_1^2 + x_0 x_2^2 + x_0^2 x_2. 
$$
Just as in Subsection \ref{subsec:BIII(D7)}, the term $f_2$ is degenerate and $f_3 (1,0,0) = 0$. 
Moreover, as in Subsection \ref{subsec:BIII(D7)}, the intersection of $f_3 = 0$ with $x_1 = 0$ is of multiplicity $k_1 = 1$. 
However, plugging $x_2 = 0$ in $f_3$ yields 
$$
    f_3 (x_0, x_1, 0) = x_0 x_1^2,
$$
which is of multiplicity $k_2 = 2$ near $x_0 = 1, x_1 = 0$. 
Lemma \ref{lem:A_k} shows that this point is of type $A_4$ and we deduce 
$$
    WH^{PIII(D8)} (q,t) = 1 + q^{2} t^2. 
$$

\subsection{Case $X = II$} \label{subsec:BII}

In this case the quadric is of the form 
$$
  Q^{PII} = - x_1 - \alpha  x_2 -x_3 + \alpha + 1 
$$
with $\alpha \in \C^{\times}$. 
We have 
$$
  F^{PII} =  x_1 x_2 x_3 - x_0^2 x_1 - \alpha x_0^2 x_2 - x_0^2 x_3 + (\alpha + 1) x_0^3,  
$$
and the singular points of $\overline{\Mod}_{\Betti}^{PII}$ over $x_0 = 0$ are $[0:1:0:0]$, $[0:0:1:0]$ and $[0:0:0:1]$. 
As in the corresponding affine co-ordinates the degree two terms are respectively given by 
$$
  x_2 x_3 - x_0^2, \quad x_1 x_3 - \alpha x_0^2, \quad x_1 x_2 - x_0^2, 
$$
and $\alpha \neq 0$, we see that all these points are of type $A_1$. As a conclusion, we get 
$$
  WH^{PII} (q,t) = 1 + q t^2 + q^{2} t^2. 
$$



\subsection{Case $X = I$} \label{subsec:BI}

In this case the quadric is of the form 
$$
  Q^{PI} = x_1 + x_2 + 1. 
$$
There are three singular points of $\overline{\Mod}_{\Betti}^{PI}$ over $x_0 = 0$: $[0:1:0:0]$, $[0:0:1:0]$ and $[0:0:0:1]$. 
At the first two of these, the degree two terms in the corresponding affine co-ordinates respectively read as 
$$
  x_2 x_3 + x_0^2, \quad x_1 x_3 + x_0^2, 
$$
so these singularities are of type $A_1$. At $[0:0:0:1]$ however, we have 
$$
 f_2 = x_1 x_2, \quad f_3 = x_0^2 x_1 + x_0^2 x_2 + x_0^3. 
$$
As $f_3 (1,0,0) = 1 \neq 0$, by virtue of Lemma \ref{lem:A_k} this singularity is of type $A_2$. In total we have three singular points, with Milnor numbers $1,1,2$ respectively, therefore 
$$
  WH^{PI} (q,t) = 1 + q^{2} t^2. 
$$

\section{Proof of Theorem \ref{thm:Simpson}}\label{sec:proof}

We start by reducing the statement to a special case, namely
the nilpotent Painlev\'e VI case, i.e. the case where the parabolic divisor 
consists of four distinct points where the Higgs field has first order poles, 
and in addition the residue of the Higgs field at each such point is nilpotent. 

\begin{prop}
 Assume that Theorem \ref{thm:Simpson} holds in the nilpotent Painlev\'e VI case. 
 Then it also holds in all cases $X$, with arbitrary choice of parameter values. 
\end{prop}

\begin{proof}
We first discuss continuous families of irregular Dolbeault spaces. 
It follows from Lemma~\ref{lem:embedding} 
and~\cite{ISS1},~\cite{ISS2},~\cite{ISS3} that for each $X$ and 
each value of the parameters of the given family, the space $\Mod_{\Dol}^{PX}$ 
is diffeomorphic to the complement of the singular fiber at infinity 
$F_{\infty}^{PX}$ in a certain elliptic fibration. 
Specifically, consider the Hirzebruch surface 
$$
    p\colon\operatorname{Tot} (K_{\CP1} (D)) \to \CP1 
$$
and its fiberwise compactification 
$$
    \operatorname{P} (\O_{\CP1} + K_{\CP1} (D))
$$
by a  section at infinity $S_{\infty}$. 
Then, there exists a birational morphism 
$$
    \varpi\colon \CP2 \# 9 \overline{\CP{}}^2 \to \operatorname{P} 
    (\O_{\CP1} + K_{\CP1} (D))
$$
resolving the base locus of a certain elliptic pencil depending on $X$ and 
the parameter values. The fiber at infinity is then given by 
$$
    F_{\infty}^{PX} = \varpi^{-1} (S_{\infty} \cup p^{-1} (D ))
$$ 
where $p^{-1} (D )$ is the scheme-theoretic fiber over $D$.

Geometrically, the choice of parameters for $\Mod_{\Dol}^{PX}$ is thus 
equivalent to that of the base locus of the associated elliptic pencil, 
and may be conveniently described 
by a point in a certain stratum $\mathcal{S}^{PX}$ of the Hilbert scheme 
\begin{equation}\label{eq:Hilb}
     \operatorname{Hilb}^8 ( \operatorname{Tot} (K_{\CP1} (D)))
\end{equation}
of $8$ points on the twisted cotangent surface of the base curve.
The parameters of the family affect $F_{\infty}^{PX}$ continuously: 
to a continuous family of parameter values there corresponds a continuous 
family of base points in $\mathcal{S}^{PX}$. 
Fixing any $X$ and arbitrarily parameter values for the system $PX$, 
it follows from connectedness of the Hilbert scheme of a connected surface that 
there exists a path 
\begin{equation}\label{eq:path}
     f\colon [0,1] \to \operatorname{Hilb}^8 ( \operatorname{Tot} (K_{\CP1} (D)))
\end{equation}
starting at the parameter values corresponding to 
the nilpotent Painlev\'e VI case and ending at the given parameter values of 
the given system $X$. 

A crucial observation is that as we have seen in Lemma~\ref{lem:b_1}, 
for any $X$ the group $H_1 ( \Mod_{\Dol}^{PX} \cap N , \Q )$ is generated 
by the normal loop $[\gamma ]$ around $S_{\infty}$. 
Said differently, we may represent a generator of the first homology 
by a loop that depends neither on $X$ nor on the values of the parameters. 

Now, we turn our attention to the Betti side: here, the group 
$H_1 ( \Mod_{\Betti}^{PX} \cap N , \Q )$ is known to be generated by a 
longitudinal loop around the cycle~\eqref{eq:nerve}, i.e. a lift of 
this cycle to the boundary of the $4$-manifold $\Mod_{\Betti}^{PX} \cap N$ 
under the map $\phi$. 
The cycle~\eqref{eq:nerve} consists of a suitable blow-up of the 
divisor~\eqref{eq:D}. 
The blow-up is only needed to turn the compactifcation of $\Mod_{\Betti}^{PX}$ 
smooth, and has no effect on the first homology group. 
Again, we see that a generator of the first homology can be given by lifting 
under $\phi$ a loop independent of $X$ and the parameter values. 
In different words, the image under $\phi$ of a generator of 
$H_1 ( \Mod_{\Betti}^{PX} \cap N , \Q )$ is independent of $X$ and the 
parameter values. 

Recall from Subsections~\ref{subsec:BVI}--\ref{subsec:BI} that the spaces 
$\Mod_{\Betti}^{PX}$ are given as certain cubic surfaces. 
To the path~\eqref{eq:path} there corresponds a continuous path of 
coefficients of the quadrics $Q^{PX}$ appearing in~\eqref{eq:cubic}. 
Now, the non-abelian Hodge and Riemann--Hilbert correspondences depend 
continuously on the parameters of the families. 
Thus, the correspondences induce homotopic maps from the generator 
of $H_1 ( \Mod_{\Dol}^{PX} \cap N , \Q )$ to that of 
$H_1 ( \Mod_{\Betti}^{PX} \cap N , \Q )$. 
Invariance of the degree for homotopies of continuous maps from $S^1$ to $S^1$ 
then shows that if $\phi \circ \psi$ induces an isomorphism on $H_1$ for the
nilpotent Painlev\'e VI case then the same holds for arbitrarily parameter 
values of any system $PX$. 
\end{proof}

There only remains to show the statement in the nilpotent Painlev\'e VI case. 
From now on we thus set $X = VI$ and we take the residue orbits to be nilpotent
with full flag parabolic structure and generic parabolic weights. 
For sake of simplicity we drop the superscript $PVI$ from the notation. 
Let $z$ and $w = z^{-1}$ be the standard charts on $\CP1$, and let $\E$ denote a holomorphic vector bundle of degree $0$ over $\CP1$. 
We consider parabolically stable logarithmic Higgs fields $\theta$ on $\E$ with singularities at $0,1,t,\infty$ having nilpotent residue at all these points. 
We have 
$$
  \mbox{tr} (\theta ) \in H^0 (\CP1, K  ) = 0, 
$$
and 
$$
  \det (\theta ) \in H^0 (\CP1, K^2  (0 + 1 + t + \infty ) ) \cong \C.
$$
The latter affine space is the Hitchin base appearing in~\eqref{eq:Hitchin_map}. We fix an isomorphism 
$$
  \O  \cong  K^2  (0 + 1 + t + \infty )
$$
given by 
$$
  1 \mapsto \frac{(\mbox{d} z)^{\otimes 2}}{z(z-1)(z-t)}. 
$$
On the other hand, we consider the holomorphic line bundle $L = K (0 + 1 + t + \infty )$ with the natural projection 
$$
  p_L : \mbox{Tot} ( L ) \to \CP1 
$$
of its total space $\mbox{Tot} ( L )$ to $\CP1$ and denote by 
$$
  \zeta \frac{\mbox{d} z}{z(z-1)(z-t)}
$$ 
the canonical section of $p_L^* L$ over $p_L^{-1} (\C)$. Then the curve 
$$
   \tilde{X}  = \{ (z, \zeta): \quad \zeta^2 + z(z-1)(z-t) = 0 \} \subset \C_z \times \C_{\zeta} 
$$
has a smooth compactification in $\mbox{Tot} ( L )$ at $z = \infty$ by the point 
$$
  w = 0, \zeta = 0. 
$$
We continue to denote this compactification by $\tilde{X}$ and moreover denote by 
\begin{align}\label{eq:double_cover}
   p: \tilde{X} & \to \CP1 \\
  (z, \zeta) & \mapsto z \notag
\end{align}
the restriction of $p_L$, a ramified double covering map with branch points $\{ 0, 1, t, \infty \}$. 
Topologically, $\tilde{X}$ is diffeomorphic to a smooth $2$-torus. 
Let us introduce the bivalued holomorphic $1$-form on $\C \setminus \{ 0,1,t \}$ 
$$
  \omega = \frac{\mbox{d} z}{\sqrt{z(z-1)(z-t)}}
$$
so that 
$$
  p^* \omega = \frac{\mbox{d} z}{\zeta}.
$$ 
Now, for instance near $(z,\zeta ) = (0,0)$ we have $z = \zeta^2 h(\zeta )$ for some holomorphic function $h$ with $h(0)\neq 0$, so we have in a neighbourhood of this point 
$$
  \frac{\mbox{d} z}{\zeta} = \frac{2 \zeta h(\zeta ) \mbox{d} \zeta +  \zeta^2 \mbox{d} h}{\zeta} = 2 h(\zeta ) \mbox{d} \zeta +  \zeta \mbox{d} h. 
$$
A similar argument near the other three ramification points show that $p^* \omega$ is a univalued holomorphic $1$-form on $\tilde{X}$, i.e. a generator of $H^0(\tilde{X}, K_{\tilde{X}})$. 

For $R>>0, \varphi \in \mathbb{R} / 2 \pi \Z$ we let $(\E_{R, \varphi} ,\theta_{R, \varphi})$ be any rank $2$ logarithmic Higgs bundle over $\CP1$ with 
$$
  \det (\theta_{R, \varphi} ) = - R e^{\sqrt{-1} \varphi} \in H^0 (\CP1, K^2  (0 + 1 + t + \infty ) )
$$
(the sign is introduced for convenience). 
We will fix $R$ and let $\varphi$ vary, and we assume that $\theta_{R, \varphi}$ depends smoothly and 
$2\pi$-periodically on $\varphi$, thus providing a smooth section of~\eqref{eq:Hitchin_map} over $|z| = R$; 
such lifts clearly exist. 
The \emph{spectral curve} of $(\E_{R, \varphi} ,\theta_{R, \varphi})$ defined as 
$$
  \tilde{X}_{R, \varphi}  = \left\{ (z, \zeta): \quad \det \left( \theta_{R, \varphi} - \zeta \frac{\mbox{d} z}{z(z-1)(z-t)} \right) = 0 \right\} \subset \mbox{Tot} ( L )
$$
is obtained by rescaling~\eqref{eq:double_cover} in the $\zeta$-direction by the factor $\sqrt{R} e^{\sqrt{-1} \varphi / 2}$. 
In particular, for any $(R, \varphi)$ the branch points of $\tilde{X}_{R, \varphi}$ are the distinct points $\{ 0, 1, t, \infty \}$, and 
the curve $\tilde{X}_{R, \varphi}$ is smooth. 
Let $z_0 \notin \{ 0,1,t, \infty \}$ and fix $\varepsilon_0 > 0$ such that $B_{2\varepsilon_0 } (z_0)$ is disjoint from $\{ 0,1,t, \infty \}$. 
According to \cite[Theorem~1.4]{Moc}, for $z\in B_{\varepsilon_0 } (z_0)$ there exists a smoothly varying frame $e_1 (z), e_2 (z)$ of $\E$ with respect to which we have 
the asymptotic equality 
$$
  \theta_{R, \varphi} (z) - \begin{pmatrix}
                             \sqrt{R} e^{\sqrt{-1} \varphi / 2} & 0 \\
                             0 & - \sqrt{R} e^{\sqrt{-1} \varphi / 2}
                            \end{pmatrix}
	\omega \to 0 
$$
as $R\to\infty$, with exponential rate. 
Because $\omega$ is bivalued, the vectors $e_1 (z), e_2 (z)$ get interchanged as the position of the point $z_0$ moves along a simple loop $\gamma$ around one of the punctures. 
In different terms the monodromy transformation of the trivialization along any such $\gamma$ is the transposition matrix 
$$
  T = \begin{pmatrix}
   0 & 1 \\
   1 & 0
  \end{pmatrix}. 
$$

We consider the connection associated by non-abelian Hodge theory to $(\E_{R, \varphi} ,\theta_{R, \varphi})$ with respect to the gauge $e_1 (z), e_2 (z)$. 
Its connection form is 
\begin{align*}
   a_{R, \varphi} (z, \bar{z}) & = \theta_{R, \varphi} (z) + \overline{\theta_{R, \varphi} (z)} + b_{R, \varphi} \\
   & \approx \sqrt{R}  
   \begin{pmatrix}
    e^{\sqrt{-1} \varphi / 2} \omega + e^{-\sqrt{-1} \varphi / 2} \bar{\omega} & 0 \\
    0 & - e^{\sqrt{-1} \varphi / 2} \omega - e^{-\sqrt{-1} \varphi / 2} \bar{\omega}
   \end{pmatrix} 
   + b_{R, \varphi} 
\end{align*}
where $\approx$ means that the difference of the two sides converges exponentially to $0$ as $R\to\infty$, 
and $b_{R, \varphi}$ is the connection form of the Chern-connection $\partial_{R, \varphi}$ of $(\E_{R, \varphi}, h_{R, \varphi} )$. 
It also follows from \cite[Theorem~1.4]{Moc} that 
\begin{equation}\label{eq:asymptotic_commuting}
   \left[ \begin{pmatrix}
    e^{\sqrt{-1} \varphi / 2} \omega + e^{-\sqrt{-1} \varphi / 2} \bar{\omega} & 0 \\
    0 & - e^{\sqrt{-1} \varphi / 2} \omega - e^{-\sqrt{-1} \varphi / 2} \bar{\omega}
   \end{pmatrix} ,
   b_{R, \varphi} \right] \to 0 
\end{equation}
exponentially as $R\to\infty$. 
Furthermore, as $\mbox{tr} (\theta ) \equiv 0$ the Higgs field induced by $\theta_{R, \varphi}$ on $\det (\E_{R, \varphi} )$ is identically zero. 
It follows that the corresponding Hermitian--Einstein metric $h_{\det (\E )} \equiv 1$  and thus $b_{R, \varphi}$ takes values in $\mathfrak{su}(2, \C)$. 
As we will only be interested in the absolute value of the integral of $a_{R, \varphi}$ along loops and the monodromy of the Chern connection is unitary, 
we will see that the actual shape of $b_{R, \varphi}$ is irrelevant for our purposes. 

In order to get a hold on $\psi (\E_{R, \varphi} ,\theta_{R, \varphi})$, we need to apply the Riemann--Hilbert correspondence to the connection obtained in the previous paragraph. 
For this purpose, we now fix $z_0$ and simple loops $\gamma_0, \gamma_1, \gamma_t$ based at $z_0$ winding about the punctures $\{ 0,1,t \}$ respectively, in positive direction. 
The monodromy matrices of the connection $\mbox{d} + a_{R, \varphi}$ associated to the punctures are given by 
\begin{equation*}
  B_j(R, \varphi) = \exp \oint_{\gamma_j} - a_{R, \varphi} (z, \bar{z})
\end{equation*}
where $j\in \{ 0,1,t \}$. Let us introduce the \emph{half-period integrals} 
$$
  \pi_j  = \oint_{\gamma_j} \omega. 
$$
By this we mean that we fix any one of the two lifts $\tilde{\gamma}_j$ of $\gamma_j$ by~\eqref{eq:double_cover} and set 
$$
  \pi_j = \int_{\tilde{\gamma}_j} p^* \omega. 
$$
By Baker--Campbell--Hausdorff formula and~\eqref{eq:asymptotic_commuting}, as $R\to\infty$ the monodromy matrix $B_j (R, \varphi)$ is asymptotically equal to 
\begin{equation*}
  T A_j(R, \varphi) \exp \sqrt{R} 
  \begin{pmatrix}
   -  e^{\sqrt{-1} \varphi / 2} \pi_j - e^{-\sqrt{-1} \varphi / 2}  \overline{\pi_j} & 0 \\
   0 & e^{\sqrt{-1} \varphi / 2} \pi_j + e^{-\sqrt{-1} \varphi / 2}  \overline{\pi_j}
  \end{pmatrix} 
\end{equation*}
where 
\begin{equation}\label{eq:unitary_factor}
  A_j(R, \varphi) \in \mbox{SU} (2 )
\end{equation}
stands for the monodromy of the Chern connection, commuting with the matrix on its right. These properties show that necessarily 
$$
   A_j(R, \varphi) = \begin{pmatrix}
                      e^{\sqrt{-1}\mu_j} & 0 \\
                      0 & e^{-\sqrt{-1}\mu_j}
                     \end{pmatrix}
$$
for some $\mu_j = \mu_j (R, \varphi) \in \mathbb{R}$. 
We then get 
\begin{equation*}
 B_j (R, \varphi) \approx 
 \begin{pmatrix}
       0 & \exp \left( - \sqrt{-1}\mu_j + 2\sqrt{R}  \Re ( e^{\sqrt{-1} \varphi / 2} \pi_j) \right) \\
       \exp \left( \sqrt{-1}\mu_j - 2\sqrt{R}  \Re ( e^{\sqrt{-1} \varphi / 2} \pi_j) \right) & 0
      \end{pmatrix} , 
\end{equation*}
and it follows that 
$$
  B_0 (R, \varphi) B_1 (R, \varphi) \approx 
  \begin{pmatrix}
   d_{01} (R, \varphi) & 0 \\
   0 & \frac 1{d_{01} (R, \varphi)} 
  \end{pmatrix}. 
$$
with 
$$
  d_{01} (R, \varphi) = \exp \left( \sqrt{-1} (\mu_1 - \mu_0) + 2\sqrt{R}  \Re ( e^{\sqrt{-1} \varphi / 2} (\pi_0 - \pi_1)) \right).
$$
Therefore, setting 
$$
  x_1 (R, \varphi) = \mbox{tr} (B_0 (R, \varphi) B_1 (R, \varphi) ) 
$$
we find 
\begin{align}
  x_1 (R, \varphi) & \approx 2 \cosh d_{01} (R, \varphi) \notag \\
   & = 2 \cosh \left( \sqrt{-1} (\mu_1 - \mu_0) + 2\sqrt{R}  \Re ( e^{\sqrt{-1} \varphi / 2} (\pi_0 - \pi_1)) \right) .\label{eq:X1}
\end{align}
Similarly, we find 
\begin{align}
 x_2 (R, \varphi) & = \mbox{tr} (B_t (R, \varphi) B_0 (R, \varphi) ) \notag \\ 
  & \approx 2 \cosh \left( \sqrt{-1} (\mu_0 - \mu_t) + 2\sqrt{R} \Re (e^{\sqrt{-1} \varphi / 2} (\pi_t - \pi_0 )) \right) \label{eq:X2} \\
 x_3 (R, \varphi) & = \mbox{tr} (B_1 (R, \varphi) B_t (R, \varphi) ) \notag \\ 
 & \approx 2 \cosh \left( \sqrt{-1} (\mu_t - \mu_1) + 2\sqrt{R} \Re ( e^{\sqrt{-1} \varphi / 2}(\pi_1 - \pi_t)) \right) \label{eq:X3}
\end{align}
where $x_2 (R, \varphi)$ and $x_3 (R, \varphi)$ are defined by the equalities in these formulas. 
It is known from \cite{FK} that these quantities fulfill the equation 
$$
  x_1 x_2 x_3 + x_1^2 + x_2^2 + x_3^2 - s_1 x_1 - s_2 x_2 - s_3 x_3 + s_4 = 0
$$
for some constants $s_1, s_2, s_3, s_4 \in \C$. By generic choice of the parabolic weights, the constants $s_m$ are generic, and the cubic 
surface determined by the above polynomial is smooth. 
As we have seen in Subsection~\ref{subsec:BVI}, in this case the compactifcation $\overline{\mathcal{M}}$
introduced in~\eqref{eq:compactifcation_Betti} agrees with the smooth compactification $\widetilde{\mathcal{M}}$ introduced in Proposition~\ref{prop:Betti}. 
The compactifying divisor is a union of three lines~\eqref{eq:D} in general position. 
The statement that the nerve complex of the boundary is homotopic to $S^1$ immediately follows. 
Let $\mathcal{N}$ stand for the dual simplicial complex of $D$. The vertices of $\mathcal{N}$ are given by 
\begin{equation}\label{eq:lines}
    v_1 = [0 : 0: x_2 : x_3], \quad v_2 = [0 : x_1 : 0 : x_3], \quad v_3 = [0:x_1 : x_2 :0],
\end{equation}
where for instance $[0:x_1 : x_2 :0]$ means the corresponding line in 
$$
  \CP2_{\infty} = \{ [0: x_1 : x_2 : x_3] \} \subset \CP3. 
$$ 
Let the edges of $\mathcal{N}$ be denoted by 
\begin{equation}\label{eq:intersection_points}
  [v_1 v_2], \quad [v_2 v_3], \quad [v_3 v_1],  
\end{equation}
respectively corresponding to the following intersection points of the divisor components: 
$$
  [0:0:0:1], \quad [0:1:0:0], \quad   [0:0:1:0]  . 
$$

 Consider open tubular neighbourhoods 
 $$
  T_1, T_2, T_3
 $$
 of the components~\eqref{eq:lines} such that 
 \begin{equation}\label{eq:triple-intersections}
   T_1\cap T_2 \cap T_3 = \emptyset.
 \end{equation}
 Set 
 \begin{equation}\label{eq:cover}
     T_{\Betti} = T_1 \cup T_2 \cup T_3, 
 \end{equation}
 so that $T_{\Betti}$ is an open tubular neighbourhood of $D$ in $\widetilde{\Mod}_{\Betti}$; $T_{\Betti}$ is a plumbed $4$-manifold. 
 Let 
 $$
 \phi_1, \phi_2, \phi_3
 $$
 be a partition of unity subordinate to the cover~\eqref{eq:cover}, and consider the map 
 \begin{align*}
    \phi : T_{\Betti} & \to \mathbb{R}^3 \\
    x & \mapsto \begin{pmatrix}
                 \phi_1 (x) \\ \phi_2 (x) \\ \phi_3 (x)
                \end{pmatrix}.
 \end{align*}
 The image of $\phi$ is contained in the standard simplex $\Delta^{2}$ of dimension $2$ because the family $\phi_j$ forms a partition of unity. 
 Moreover, it follows from~\eqref{eq:triple-intersections} that 
 $$
  \im(\phi ) \subset \left( \Delta^{2} \right)^1, 
 $$
 the $1$-skeleton of $\Delta^{2}$. 
 A closer look shows that we have the equality: 
 \begin{equation}\label{eq:ImPhi}
    \im(\phi ) = \left( \Delta^{2} \right)^1 = [v_1 v_2] \cup [v_2 v_3] \cup [v_3 v_1], 
 \end{equation}
 which is homotopy equivalent to $S^1$.

We need to show that for any section $(\E_{R, \varphi} ,\theta_{R, \varphi})$ of $h$ over $|z| = R$, the loop 
$$
  \phi \circ \psi (\E_{R, \varphi} ,\theta_{R, \varphi})
$$
is a generator of $\pi_1 (|\mathcal{N} |)$. 
For this purpose, we need to study the asymptotic behaviour of the quotients 
$$
  \frac{x_1(R, \varphi)}{x_2(R, \varphi)}, \quad \frac{x_3(R, \varphi)}{x_1(R, \varphi)}, \quad \frac{x_2(R, \varphi)}{x_3(R, \varphi)}
$$
as $R\to +\infty$, and in particular the way this behaviour depends on $\varphi \in [0, 2\pi ]$. 
For this purpose observe first that for $d\in \C$ with  
$|\Re (d)| >> 0$ we have 
$$
  |2 \cosh (d ) | \approx e^{| d |}.
$$
Applying this asymptotic equivalence to~\eqref{eq:X1},~\eqref{eq:X2},~\eqref{eq:X3} for $R>>0$ we find 
\begin{align*}
 |x_1(R, \varphi)| & \approx  \exp \left( 2\sqrt{R} |\Re ( e^{\sqrt{-1} \varphi / 2} ( \pi_0 - \pi_1))|\right), \\
 |x_2(R, \varphi)| & \approx  \exp \left( 2\sqrt{R} |\Re ( e^{\sqrt{-1} \varphi / 2} ( \pi_t - \pi_0))|\right), \\
 |x_3(R, \varphi)| & \approx  \exp \left( 2\sqrt{R} |\Re ( e^{\sqrt{-1} \varphi / 2} ( \pi_1 - \pi_t))|\right).  
\end{align*}
For generic $t\in \CP1 \setminus \{ 0, 1 , \infty \}$ the periods $\pi_0, \pi_1, \pi_t$ are not colinear in $\C$, 
said differently they form the vertices of a non-degenerate triangle $\Delta$ with sides 
$$
  a = \pi_0 - \pi_1, \quad   b = \pi_t - \pi_0, \quad c = \pi_1 - \pi_t.
$$
Consider the triangles $e^{\sqrt{-1} \varphi / 2} \Delta$ as $\varphi$ ranges over $[0, 2\pi )$. 
A straightforward geometric inspection shows that the lengths of the projection onto the real axis 
of the three sides of $e^{\sqrt{-1} \varphi / 2} \Delta$ obey the following rule. 
\begin{lem}\label{lem:triangle}
Let $\Delta \subset \C$ be any non-degenerate triangle with sides $a,b,c \in \C$ such that $a + b + c = 0$. 
Let us denote by $e^{\sqrt{-1} \varphi / 2} \Delta$ the triangle obtained by rotating $\Delta$ by angle $\varphi / 2$ in the positive direction, with sides 
$e^{\sqrt{-1} \varphi / 2} a, e^{\sqrt{-1} \varphi / 2} b, e^{\sqrt{-1} \varphi / 2}c$. 
Then, for each side $a,b,c$ there exists exactly one value $\varphi_a, \varphi_b, \varphi_c \in [0, 2\pi )$ such that $e^{\sqrt{-1} \varphi_a / 2} a$ 
(respectively $e^{\sqrt{-1} \varphi_b / 2}b, e^{\sqrt{-1} \varphi_c / 2}c$) is purely imaginary. 
We have 
$$
  \Re ( e^{\sqrt{-1} \varphi_a / 2} b ) = \Re ( e^{\sqrt{-1} \varphi_a / 2} c ). 
$$
In addition, the function 
$$
  \Re ( e^{\sqrt{-1} \varphi / 2} b ) - \Re ( e^{\sqrt{-1} \varphi / 2} c )
$$
changes sign at $\varphi = \varphi_a$. Similar statements hold with $a,b,c$ permuted. 
\end{lem}
We call $\varphi_a, \varphi_b, \varphi_c$ the \emph{critical angle} of the sides $a,b,c$ respectively. 
By genericity, the critical angles $\varphi_a, \varphi_b, \varphi_c$ are pairwise different. 
It follows from the lemma that the critical angles decompose $S^1$ into three closed arcs 
$$
  S^1 = I_1 \cup I_2 \cup I_3 
$$
pairwise intersecting each other in a critical angle, satisfying the property: 
$$
  \max (|\Re (e^{\sqrt{-1} \varphi / 2} (\pi_0 - \pi_1 ))|, |\Re (e^{\sqrt{-1} \varphi / 2} (\pi_t - \pi_0 ))|, |\Re ( e^{\sqrt{-1} \varphi / 2}(\pi_1 - \pi_t))| )
$$
is realized
\begin{itemize}
 \item by $|\Re (e^{\sqrt{-1} \varphi / 2} (\pi_0 - \pi_1 ))|$ for $\varphi \in I_1$, 
 \item by $|\Re (e^{\sqrt{-1} \varphi / 2} (\pi_t - \pi_0 ))|$ for $\varphi \in I_2$, 
\item and by $|\Re (e^{\sqrt{-1} \varphi / 2} (\pi_1 - \pi_t ))|$ for $\varphi \in I_3$.
\end{itemize}
Namely, $I_1$ is the arc with end-points $\varphi_b, \varphi_c$ not containing $\varphi_a$, and so on. 
Let us denote by $\mbox{Int} ( I )$ the interior of an arc $I \subset S^1$, and for ease of notation let us set 
$x_j = X_j(R, \varphi)$.
It follows that  as $R \to + \infty$ 
\begin{itemize}
 \item for $\varphi \in \mbox{Int} ( I_1 )$, we have 
 \begin{align*}
  \frac{x_1}{x_2} & \to \infty,  & \frac{x_1}{x_3} & \to \infty, & 
  [x_0 : x_1 : x_2 : x_3] & \to [0:1:0:0]
 \end{align*}
\item for $\varphi \in \mbox{Int} (I_2 )$, we have 
 \begin{align*}
  \frac{x_2}{x_1} & \to \infty, & \frac{x_2}{x_3} & \to \infty, & 
  [x_0 : x_1 : x_2 : x_3] & \to [0:0:1:0],
 \end{align*}
 \item 
 for $\varphi \in \mbox{Int} ( I_3 )$, we have 
 \begin{align*}
  \frac{x_3}{x_1} & \to \infty, & \frac{x_3}{x_2} & \to \infty, &
  [x_0 : x_1 : x_2 : x_3] & \to [0:0:0:1],
 \end{align*} 
\end{itemize}
all convergence rates being exponential in $\sqrt{R}$. 
These limits show that 
\begin{itemize}
 \item for $\varphi \in \mbox{Int} ( I_1 )$, we have 
 $$
  \phi_1 \circ \psi (\E_{R, \varphi} ,\theta_{R, \varphi}) = 0, 
 $$
 \item for $\varphi \in \mbox{Int} (I_2 )$, we have 
 $$
  \phi_2 \circ \psi (\E_{R, \varphi} ,\theta_{R, \varphi}) = 0, 
 $$
 \item for $\varphi \in \mbox{Int} ( I_3 )$, we have 
 $$
  \phi_3 \circ \psi (\E_{R, \varphi} ,\theta_{R, \varphi}) = 0. 
 $$
\end{itemize}
Said differently, 
\begin{itemize}
 \item for $\varphi \in \mbox{Int} ( I_1 )$, we have 
 $$
  \phi \circ \psi (\E_{R, \varphi} ,\theta_{R, \varphi}) \in [v_2 v_3], 
 $$
 \item for $\varphi \in \mbox{Int} (I_2 )$, we have 
 $$
  \phi \circ \psi (\E_{R, \varphi} ,\theta_{R, \varphi}) \in [v_3 v_1], 
 $$
 \item for $\varphi \in \mbox{Int} ( I_3 )$, we have 
 $$
  \phi \circ \psi (\E_{R, \varphi} ,\theta_{R, \varphi}) \in [v_1 v_2].  
 $$
\end{itemize}
We infer that as 
$\varphi$ ranges over $[0,2\pi ]$ 
the corresponding elements 
$$
  \phi \circ \psi (\E_{R, \varphi} ,\theta_{R, \varphi}) \in {\Mod}_{\Betti}
$$ 
describe a path which is a generator of $\pi_1 (| \mathcal{N}  |) \cong \Z$. 
This finishes the proof of Theorem~\ref{thm:Simpson}.  

\section{Matching the filtrations in the Painlev\'e VI case}\label{sec:filtrations}

We keep the assumptions of Section~\ref{sec:proof}, in particular the moduli spaces we consider are the ones corresponding to the Painlev\'e VI case, and we drop the superscripts $PVI$.

 \begin{lem}\label{lem:generators}
 The diffeomorphism 
 $$
  \Mod_{\Dol} \to \Mod_{\Betti}  , 
 $$
 maps a generator of $\Gr^P_{2} H^2 (\Mod_{\Dol}, \C )$ to a generator of $\Gr^W_{4} H^2 (\Mod_{\Betti}, \C )$ . 
 \end{lem}

\begin{proof}
 We work dually in homology. 
 Denote the class of the generic Hitchin fiber by 
 $$
  [HF] = [h^{-1}(Y_{-1})] \in H_2 (\Mod_{\Dol}, \Q), 
 $$ 
 where $Y_{-1} \in Y = \C$ is a point in the Hitchin base; by~\eqref{eq:generatorP1} it is dual to a generator of $\Gr^P_{2} H^2 (\Mod_{\Dol}, \Q)$. 
 Let $U$ be an affine open neighbourhood of $[0:1:0:0]$ in $\overline{\Mod}_{\Betti}$ and 
 $$
  z_1 = r_1 e^{i\theta_1}, \quad  z_2 = r_2 e^{i\theta_2}
 $$
 coordinates on $U$ defining the two divisors $v_2, v_3$ crossing at $[0:1:0:0]$. 
 Fixing a value $R>>0$, it follows from the analysis of Section~\ref{sec:proof} that for $\varphi\in I_1$ the image of the Hitchin fiber $h^{-1} (R e^{\sqrt{-1} \varphi} )$ is contained in a 
 tubular neighbourhood of the torus in $U \cap {\Mod}_{\Betti}$ defined by 
 \begin{equation}\label{eq:Betti_torus}
    C = \{ r_1 = \varepsilon_1, r_2 = \varepsilon_2 \}
 \end{equation}
 for some small constants $0 < \varepsilon_1,  \varepsilon_2 << 1$. 
 Clearly, $U \cap {\Mod}_{\Betti}$ deformation retracts onto $C$, so we have 
 $$
  H_2 (U \cap {\Mod}_{\Betti}, \Q) \cong \Q. 
 $$
 Since $h^{-1} (R e^{\sqrt{-1} \varphi} )$ is not a boundary in $\Mod_{\Dol}$, it follows that its image in ${\Mod}_{\Betti}$ is homologous to a non-zero rational 
 multiple of the fundamental class of the torus $C$ in~\eqref{eq:Betti_torus}: 
 \begin{equation}\label{eq:generators}
  [HF] = q [C]
 \end{equation}
 for some $q \in \Q^{\times}$. 

 On the other hand, up to terms containing at most one logarithmic factor locally near $[0:1:0:0]$ a generator $[\eta ]$ of $\Gr^W_{4} H^2 (\Mod_{\Betti}^{PX}, \C )$ may be represented by a $2$-form 
 \begin{align}\label{eq:weight4form}
   \eta & = f(z_1, \bar{z}_1, z_2, \bar{z}_2 ) \frac{\mbox{d} z_1}{z_1} \wedge \frac{\mbox{d} z_2}{z_2} \\ 
   & = f \left( \frac{\mbox{d} r_1}{r_1} \wedge \frac{\mbox{d} r_2}{r_2} + i \frac{\mbox{d} r_1}{r_1} \wedge \mbox{d} \theta_2  
   + i \mbox{d} \theta_1 \wedge \frac{\mbox{d} r_2}{r_2} - \mbox{d} \theta_1 \wedge \mbox{d} \theta_2 \right) \notag
 \end{align}
 for some smooth function $f$ on $U$. 
 Now, the last term in the expression of the right-hand side of~\eqref{eq:weight4form} evaluates to $-4  \pi^2 f (0,0)$ on $C$, and the other terms vanish on it. 
 In the spectral sequence ${}_W E_2$ the term $\Ker (\delta )$ (giving rise to $W_{4} H^2 (\Mod_{\Betti}, \C )$) is generated by the triple $(1, 1 , 1)$, which means that for a generator $\eta$ we must have 
 $$
  f (0,0) \neq 0. 
 $$
 This finishes the proof.
 \end{proof}

Using the above preparatory results, we are ready to prove that the filtrations $P$ and $W$ match under non-abelian Hodge theory. 
For this purpose, by Propositions~\ref{prop:Dol},~\ref{prop:Betti} and equations~\eqref{eq:generatorP1} and~\eqref{eq:generatorW2} we need to show that the composition 
$$
  H^2 \left( \widetilde{\Mod}_{\Betti}, \C \right) \xrightarrow{i^*} H^2 (\Mod_{\Betti}, \C ) = H^2 (\Mod_{\Dol}, \C )  \xrightarrow{j^*}   H^2  \left( h^{-1}(R e^{\sqrt{-1} \varphi}), \C \right) 
$$
is the $0$-map. The dual statement in homology is that the morphism 
$$
  H_2 \left( h^{-1}(R e^{\sqrt{-1} \varphi}), \C \right) \to H_2 \left( \widetilde{\Mod}_{\Betti}, \C \right)
$$
induced by inclusion vanishes. Since $H_2 (h^{-1}(R e^{\sqrt{-1} \varphi}), \C )$ is generated by the fundamental class of $h^{-1}(R e^{\sqrt{-1} \varphi})$, it is sufficient to 
show that the image of $h^{-1}(R e^{\sqrt{-1} \varphi})$ is a $2$-boundary in $\widetilde{\Mod}_{\Betti}$. By~\eqref{eq:generators}, it is sufficient to show that $C$ is a $2$-boundary in $\widetilde{\Mod}_{\Betti}$. 
This latter assertion is easy to show: with the notations of Lemma~\ref{lem:generators} we have 
$$
  C = \partial \left( \{ r_1 = \varepsilon \} \times  \{ r_2 \leq \varepsilon \} \right). 
$$



\begin{thebibliography}{50}
 \bibitem{AF} D.\;Arinkin, R.\;Fedorov, {\it An example of the Langlands correspondence for irregular rank two connections on $\mathbb{P}^1$}, Adv. Math. (3) {\bf 230} (2012)
 \bibitem{BKV} D.\;Baraglia, M.\;Kamgarpour, R.\;Varma, {\it Complete integrability of the parahoric Hitchin system}, Int. Math. Res. Not. {\bf 21} (2019)
 \bibitem{BHPV} W.\;P.\;Barth, K.\;Hulek, C.\;A.\;M.\;Peters, A.\; van de Ven, {\it Compact Complex Surfaces} vol. {\bf 4}, Ergebnisse der Mathematik und ihrer Grenzgebiete 3. Folge, A Series of Modern Surveys in Mathematics, Springer (2004)
 \bibitem{BBD} A.\;Beilinson, J.\;Bernstein, P.\;Deligne, {\it Faisceaux pervers}, Ast\'erisque {\bf 100} (1982)
 \bibitem{BB} O.\;{Biquard}, P.\;{Boalch}, {\it Wild non-abelian {H}odge theory on curves}, Compos. Math. (1) {\bf 140} (2004), 179--204.
 \bibitem{Boalch1} P.\;{Boalch}, {\it Quasi-Hamiltonian Geometry of Meromorphic Connections}, Duke Math. J. (2) {\bf 139} (2007), 369--405.
 \bibitem{Boalch2} P.\;{Boalch}, {\it Geometry and braiding of Stokes data; Fission and wild character varieties}, Ann. Math. {\bf 179} (2014), 301--365.
 \bibitem{Boalch3} P.\;{Boalch}, {\it Symplectic manifolds and isomonodromic deformations}, Adv. Math. (2) {\bf 163}, 137--205. (2001)
 \bibitem{BT} R.\;Bott, L.\;Tu, {\it Differential forms in Algebraic Topology}, Graduate Texts in Mathematics {\bf 82}, Springer Verlag
 \bibitem{Bottacin} F.\;Bottacin, {\it Symplectic geometry on moduli spaces of stable pairs}, Ann. Sci. \'Ecole Norm. Sup. (4) {\bf 28} (1995) 391--433.
 \bibitem{BW} J.\;Bruce, C.\;T.\;C.\;Wall, {\it On the classification of cubic surfaces}, J. London Math. Soc. (2), {\bf 19} (1979), 245--256.
 \bibitem{CDDP}  W.\;Chuang, D.\;Diaconescu, R.\;Donagi, T.\;Pantev, {\it Parabolic refined invariants and Macdonald polynomials}, Comm. Math. Phys. {\bf 335}, (2015), 1323--1379.
 \bibitem{HdCM} M.\;de\;Cataldo, T.\;Hausel, L.\;Migliorini, {\it Topology of Hitchin systems and Hodge theory of character varieties: the case $A_1$}, Ann. Math. (3) {\bf 175} (2012), 1329--1407.
 \bibitem{dCM} M.\;de\;Cataldo, L.\;Migliorini, {\it The perverse filtration and the Lefschetz hyperplane theorem}, Ann. Math. (3) {\bf 171} (2010), 2089--2113.
 \bibitem{DelHodge2} P.\;Deligne, {\it Th\'eorie de Hodge: II}, Publ. Math. I.H.\'E.S. {\bf 40} (1971), 5--57. 
 \bibitem{FK} R.\;Fricke and F.\;Klein, {\it Vorlesungen \"uber die Theorie der automorphen Funktionen; Die gruppentheoretischen Grundlagen}, Teubner (1897)
 \bibitem{HMW} T.\;Hausel, M.\;Mereb, M.\;Wong, {\it Arithmetic and representation theory of wild character varieties}, J. Eur. Math. Soc. {\bf 21} (2016)
 \bibitem{HT} T.\;Hausel, M.\;Thaddeus, {\it Mirror symmetry, Langlands duality, and the Hitchin system}, Inv. Math. (1) {\bf 153} (2003), 197--229.
 \bibitem{HRLT} D.\;Hernandez Ruip\'erez, A.\;C.\;L\'opez Mart\'{\i}n, D.\;Sanch\'ez G\'omez, C.\;Tejero Prieto, {\it Moduli spaces of semistable sheaves on singular genus $1$ curves}, Int. Math. Res. Not. {\bf 23} (2009), 4428--4462.
 \bibitem{HST} S.\;Hosono, M.\;Saito, A.\;Takahashi, {\it Relative Lefschetz action and BPS state counting}, Int. Math. Res. Not. {\bf 15} (2001), 783--816.
 \bibitem{ISS1} P.\;Ivanics, A.\;Stipsicz, Sz.\;Szab\'o, {\it Two-dimensional moduli spaces of rank $2$ {H}iggs bundles over $\mathbb{CP}^{1}$ with one irregular singular point}, 
    J. Geom. Phys. {\bf 130} (2018), 184--212.
 \bibitem{ISS2} P.\;Ivanics, A.\;Stipsicz, Sz.\;Szab\'o, {\it Hitchin fibrations on moduli of irregular Higgs bundles and motivic wall-crossing}, 
    Journal of Pure and Applied Algebra (9) {\bf 223} (2019) 3989--4064. 
 \bibitem{ISS3} P.\;Ivanics, A.\;Stipsicz, Sz.\;Szab\'o, {\it Hitchin fibrations on moduli of irregular Higgs bundles with one singular fiber}, Symmetry,
 Integrability and Geometry: Methods and Applications (SIGMA) (85) {\bf 15} (2019)
 \bibitem{KNPS} L.\;Katzarkov, A.\;Noll, P.\;Pandit and C.\;Simpson, {\it Harmonic maps to buildings and singular perturbation theory} Commun. Math. Phys. {\bf 336}, No. 2, (2015) 853--903. 
 \bibitem{Kod} K.\;Kodaira, {\it On compact analytic surfaces: {II}}, Ann. Math. {\bf 77} (1963), 563--626.
 \bibitem{Komyo} A.\;Komyo, {\it On compactifcations of character varieties of $n$-punctured projective line}, Ann. Inst. Fourier (Grenoble), {\bf 65} (2015), 1493--1523.
 \bibitem{Mal} B.\;Malgrange, {\it \'Equations differentielles \`a coefficients polynomiaux}, Progress in Mathematics {\bf 96} (1991), Birkh\"auser
 \bibitem{Markman} E.\;Markman, {\it Spectral curves and integrable systems}, Compos. Math. {\bf 93} (1994), 255--290.
 \bibitem{Moc} T.\;Mochizuki, {\it Asymptotic behaviour of certain families of harmonic bundles on Riemann surfaces} J. Topol. {\bf 9}, No. 4, (2016) 1021--1073.
 \bibitem{N-Sz} A.\;N\'emethi, Sz.\;Szab\'o, {\it The Geometric P=W conjecture in the Painlev\'e cases via plumbing calculus}, Int. Math. Res. Not. 
 \texttt{https://doi.org/10.1093/imrn/rnaa245}
 \bibitem{O} P.\;Orlik,  {\it Seifert manifolds}, Lecture Notes in Mathematics {\bf 291}, Springer (1972) 
 \bibitem{PS} M.\;van\;der\;Put, M.\;Saito, {\it Moduli spaces of linear differential equations and the Painlev\'e equations}, Annales Inst. Fourier (Grenoble) (7), {\bf 59} (2009), 2611--2667.
 \bibitem{Sh1} T.\;Shioda, {\it On elliptic modular surfaces}, J. Math. Soc. Japan {\bf 24} (1972), 20--59.
 \bibitem{Sh2} T.\;Shioda, {\it On the Mordell-Weil Lattices}, Commentarii Mathematici Universitatis Sancti Pauli, (2) {\bf 39} (1990), 211--239.
 \bibitem{Sim} C.\;Simpson, {\it The dual boundary complex of the $SL_2$ character variety of a punctured sphere}, Ann. Fac. Sci. Toulouse, Math. (6) {\bf 25}, No. 2-3, Part A (2016), 317--361.
 \bibitem{Sim_Hodge} C.\;Simpson, {\it Harmonic bundles on noncompact curves}, J. Amer. Math. Soc., (3) {\bf 3} (1990), 713--770.
 \bibitem{SSS} A.\;Stipsicz, Z.\;Szab\'o, \'A.\;Szil\'ard, {\it Singular fibers in elliptic fibrations on the rational elliptic surface}, Periodica Mathematica Hungarica, {\bf 54} (2007) 137--162.
 \bibitem{Sz_BNR} Sz.\;Szab\'o, {\it The birational geometry of unramified irregular {H}iggs bundles on curves}, Intern. J. Math. (6), {\bf 28} (2017)
 \bibitem{Tate} J.\;Tate, {\it On the conjectures of Birch and Swinnerton-Dyer and a geometric analog}, S\'eminaire Bourbaki {\bf 9}, Soc. Math. France, Paris, 1966, Exp. No. 306, (1995) 415--440.
 \bibitem{Zhang} Z.\;Zhang, {\it Multiplicativity of perverse filtration for {H}ilbert schemes of fibered surfaces}, Adv. Math. {\bf 312} (2017), 636--679.
\end{thebibliography}
\end{document}